\documentclass[11pt]{article}
    
\usepackage{amsmath,amsfonts,amsthm,amscd,amssymb,graphicx}

\usepackage{subfigure}

\numberwithin{equation}{section}
\usepackage{hyperref}

\newtheorem{theorem}{Theorem}[section]
\newtheorem{theo}{Theorem}[section]

\newtheorem{proposition}[theorem]{Proposition}
\newtheorem{prop}[theorem]{Proposition}

\def\eps{\varepsilon }

\newcommand{\RR}{\mathbb{R}}
\newcommand{\cO}{\mathcal{O}}

\newcommand{\ZZ}{{\mathbb Z}}

\def\beq{\begin{equation}}
\def\eeq{\end{equation}}
\def\bb1{{1\!\!1}}
\def\cB{\mathcal{B}}

\def\rit{{\Bbb R}}

\def\zit{{\Bbb Z}}

\def\eps{\varepsilon}

\def\cA{{\cal A}}
\def\cB{{\cal B}}

\begin{document}

\title{$L^\infty$ instability of Prandtl layers}

\author{Emmanuel Grenier\footnotemark[1]
  \and Toan T. Nguyen\footnotemark[2]
}

\maketitle

\renewcommand{\thefootnote}{\fnsymbol{footnote}}

\footnotetext[1]{Equipe Projet Inria NUMED,
 INRIA Rh\^one Alpes, Unit\'e de Math\'ematiques Pures et Appliqu\'ees., 
 UMR 5669, CNRS et \'Ecole Normale Sup\'erieure de Lyon,
               46, all\'ee d'Italie, 69364 Lyon Cedex 07, France. Email: Emmanuel.Grenier@ens-lyon.fr}

\footnotetext[2]{Department of Mathematics, Penn State University, State College, PA 16803. Email: nguyen@math.psu.edu. TN's research was partly supported by the NSF under grant DMS-1764119 and by an AMS Centennial Fellowship.}

% \tableofcontents

\subsubsection*{Abstract}

%%%%

In $1904$, Prandtl introduced his famous boundary layer in order to describe the behavior of solutions of incompressible Navier Stokes
equations near a boundary as the viscosity goes to $0$. His Ansatz was that the solution of Navier Stokes equations
can be described as a solution of Euler equations, plus a boundary layer corrector, plus
a vanishing error term in $L^\infty$ in the inviscid limit. In this paper we prove that, 
for a class of  smooth solutions of Navier Stokes equations, namely for shear layer profiles which are unstable for
Rayleigh equations,  this Ansatz is false if we consider solutions with Sobolev regularity, in strong contrast with the analytic case,
pioneered by R.E. Caflisch and  M. Sammartino \cite{SammartinoCaflisch1,SammartinoCaflisch2}.

Meanwhile we address the classical problem of the nonlinear stability of shear layers near a boundary and
prove that if a shear flow is spectrally unstable for Euler equations, then it is non linearly unstable for the Navier Stokes
equations provided the viscosity is small enough.

%%%%%%%%%%%%%%%%%%%%%%%%

\section{Introduction}

%%%%%%%%%%%%%%%%%%%%%%%%

In this paper we address the question of the description of solutions of incompressible Navier Stokes equations in a bounded domain,
in the case of the zero Dirichlet boundary condition. More precisely, let $\Omega$ be the half plane $x \in \rit$, $y > 0$.
Let $u^\nu$ be solutions of incompressible Navier Stokes equations with forcing term $f^\nu$
\beq \label{NS1}
\partial_t u^\nu + (u^\nu \cdot \nabla) u^\nu - \nu \Delta u^\nu + \nabla p^\nu = f^\nu,
\eeq
\beq \label{NS2}
\nabla \cdot u^\nu = 0
\eeq
and Dirichlet boundary condition
\beq \label{NS3} 
u^\nu = 0 \qquad \hbox{on} \qquad y = 0 .
\eeq
As the viscosity goes to $0$, we expect $u^\nu$ to converge to a solution of Euler equations
\beq \label{Eu1}
\partial_t u^E + (u^E  \cdot \nabla ) u^E  + \nabla p^E  = f^0,
\eeq
\beq \label{Eu2}
\nabla\cdot  u^E  = 0 
\eeq
with boundary condition
\beq \label{Eu3}
u_2^E  = 0 \qquad \hbox{on} \qquad y = 0 .
\eeq
The justification of this convergence is however very delicate, since the boundary conditions dramatically change. 
As a consequence, a boundary
layer is expected near $y = 0$ in order to describe the transition between Navier Stokes boundary conditions and Euler boundary conditions.
To take into account this transition, Prandtl \cite{Pra:1904} introduced the following Ansatz
\beq \label{Ansatz}
u^\nu(t,x,y) = u^E(t,x,y) + u^P(t,x,y / \sqrt{\nu})  + o(1)_{L^\infty},
\eeq
where $u^P$ describes the behavior of $u^\nu$ in a boundary layer of size $\cO(\sqrt \nu)$, 
called the Prandtl's boundary layer, and the remainder $o(1)_{L^\infty}$ tends to zero in the inviscid limit. The boundary layer corrector $u_1 = u_1^E(t,x,0)+ u_1^P(t,x,Y)$ is then constructed by solving the classical Prandtl boundary layer equation
\beq\label{Prandtl}
\begin{aligned}
	\partial_t u_1 + u_1 \partial_x u_1 + u_2 \partial_Y u_1 &= \partial_Y^2 u_1 - \partial_x p^E(t,x,0) 
	\\ \partial_xu_1 + \partial_Y u_2 & =0
\end{aligned}   
\eeq 
together with the no-slip boundary conditions $u_1 = u_2 = 0$ at $Y=0$ and the matching condition $u_1(t,x,Y) \to u^E_1(t,x,0)$ as $Y \to \infty$. 

%Observe that the pressure $p^E$ is a known function in term of solutions to Euler %equations \eqref{Eu1}-\eqref{Eu3}. 

The existence and uniqueness of solutions to the Prandtl equations 
have been constructed for monotonic data by Oleinik  \cite{Ole} in the sixties. There are also recent reconstructions \cite{Alex, MW} of Oleinik's solutions via a more direct energy method. For data with analytic or Gevrey regularity, the well-posedness of the Prandtl equations is established in \cite{SammartinoCaflisch1, GVMasmoudi15}, among others. In the case of non-monotonic data with Sobolev regularity, the Prandtl boundary layer equations are known to be ill-posed (\cite{GVDormy,GVN1,GN1}). 

Concerning the validity of Prandtl's Ansatz \eqref{Ansatz}, this was established for data with analytic regularity in the celebrated work of Caflisch and Sammartino \cite{SammartinoCaflisch2}. In particular, it was proven that if a boundary layer Ansatz exists to describe the limiting
behavior of $u^\nu$, then it must be of the Prandtl's form \eqref{Ansatz}. 
%with a remainder of order $\mathcal O(\sqrt \nu)$. 
A similar result were also obtained by \cite{Mae} for data whose initial vorticity is compactly supported away from the boundary. The stability of shear flows under perturbations with Gevrey regularity is recently proved in \cite{GVM}.

However, considering analytic or Gevrey initial data is too restrictive, since it precludes small
but high frequencies perturbations, which are more physically relevant. The first author proved in \cite{Grenier00CPAM} that the Ansatz \eqref{Ansatz} is nonlinearly unstable with a vanishing lower bound of order $\mathcal O(\nu^{1/4})$ for the remainder. Up to now, there were no result which proved, or disproved, 
the Ansatz
(\ref{Ansatz}) for data with Sobolev regularity. %XXX Grenier 2000

%, despite several efforts \cite{GVM,Grenier00CPAM,GrN5,GrN2,Mae}.

In this paper we give the first result in this direction. Namely we prove that there exists particular initial data such that (\ref{Ansatz})
is {\it wrong}. More precisely we will show that some shear layer profiles are nonlinearly unstable,  for these profiles
the remainder in \eqref{Ansatz} 
reaches order one in the inviscid limit. Proving the instability of order one, or in fact any order beyond $\mathcal O(\nu^{1/4})$ obtained in \cite{Grenier00CPAM}, faces a serious obstruction: the viscous boundary sublayers which arise from the instability of the main Prandtl's layer are themselves unstable, giving rise to thinner and thinner viscous sublayers. The instability of these thinner sublayers is inevitable due to the linear instability theory of generic shear flows \cite{Reid,GGN3} and the fact that the local Reynolds is of order $\frac{U_{sub}}{\nu^{1/4}} \to \infty$ whenever the amplitude of sublayers $U_{sub}$ goes beyond $\nu^{1/4}$. As a consequence, there are many instabilities from both the main Prandtl's layers and the sublayers, and it remains unclear which sublayers are dominant in the large time. This is the main limitation of the previous method \cite{Grenier00CPAM}. For more details of the obstruction, see Section \ref{sec-strategy}. See also \cite{GrN3} for a further link between the stability of classical Prandtl's layers and that of viscous sublayers. 

This paper not only proves the invalidity of the Ansatz \eqref{Ansatz}, but also constructs a three-layer solution to Navier-Stokes equations involving an Euler flow (trivial), a classical Prandtl's layer with thickness of order $\sqrt{\nu}$, and a thinner boundary sublayer with thickness of order $\nu^{3/4}$. This latter sublayer in turn gives rise to thinner sub-sublayers with thickness of order $\nu^{7/8}$, which is confirmed linearly \cite{GGN3}. This paper builds the first step towards fully justifying the boundary layer cascade developed near the boundary. 

Let us mention that if one replaces the classical no-slip boundary condition \eqref{NS3} by a Navier-slip condition, the boundary layers are less violent with a much smaller amplitude of order $\sqrt{\nu}$. As a consequence, the inviscid limit and the boundary layer Ansatz are established in this case: see for instance \cite{IftimieSueur,MasRou1}. The instability observed in this paper does not apply to these settings. However, when the slip length is of order $\sqrt{\nu}$ or smaller, a similar instability up to order one can be obtained via an energy method \cite{Paddick}, adapted from \cite{Grenier00CPAM}. A crucial difference between the slip and no-slip boundary conditions is that there are intricate (and unstable) boundary sublayers arising in the latter case, but not in the former.

\subsection{Main results}

Let us now detail our main results of this paper. A shear layer profile is a solution of the form
$$
U^\nu(t,x,y) = \begin{pmatrix}U(t,y/ \sqrt{\nu}) \\ 0 \end{pmatrix} 
$$
that is a solution of both Prandtl and Navier Stokes equations. Here $U(t,y)$ is a smooth function with $U(t,0) = 0$ such that $U(t,y)$ converges when $y \to + \infty$
to a constant Euler flow $U_\infty$. In this paper we consider two cases

\begin{itemize} 

\item time dependent boundary layers: 
%in this case, there is no forcing term $F^\nu = 0$ and
$$
\partial_t U^\nu - \partial_{yy} U^\nu = 0,
$$
namely $U^\nu$ is a solution of the classical heat equation

\item time independent boundary layers: in this case, we add  a time-independent forcing term which compensates for
the viscosity. Precisely, we take 
\begin{equation}\label{def-source}
F^\nu = \begin{pmatrix} - U''(0,y / \sqrt{\nu}) \\0\end{pmatrix} .
\end{equation}

\end{itemize}

The main result of this paper is as follows. 
\begin{theo} \label{maintheo}
There exists a smooth, analytic function $U(0,Y)$, such that the corresponding  sequence of  time dependent shear layers 
\beq \label{shear}
U^\nu(t,x,y) = \begin{pmatrix} U(t,y / \sqrt{\nu}) \\0 \end{pmatrix},
\eeq
which are smooth solutions
of Navier Stokes, Prandtl, and heat equations satisfies the following assertion. For any $N$ and $s$ arbitrarily 
large, there exist $\sigma_0 > 0$, $C_0 > 0$ 
and a sequence of solutions $u^\nu$ of Navier Stokes equations \eqref{NS1}-\eqref{NS3} with forcing terms $f^\nu$,
 on some interval $[0,T^\nu]$, such that
$$
\| u^\nu(0) - U^\nu(0) \|_{H^s} \le \nu^N,
$$
$$
\| f^\nu \|_{L^\infty([0,T^\nu],H^s)} \le \nu^N,
$$
but
$$
\| u^\nu(T^\nu) - U^\nu(T^\nu) \|_{L^\infty} \ge \sigma_0
$$
and
$$
T^\nu = O ( \sqrt{\nu} \log \nu^{-1} ).
$$
\end{theo}

This theorem proves that Prandtl Ansatz is false in $L^\infty$ in very small times, of order $\sqrt\nu \log \nu^{-1}$. The same theorem holds true for the time independent boundary layer with a forcing term (\ref{def-source}). In particular, it is proved that the convergence of Navier-Stokes solutions to Euler solutions, plus a boundary layer, fails in $L^\infty$ in the inviscid limit. We remark that this however does not prevent the convergence to hold in $L^p$ for $p<\infty$. 

In addition, as will be clear from the construction, the Navier-Stokes solutions obtained in Theorem \ref{maintheo} involve not only the Prandt's layer of size $\sqrt\nu$, but also a viscous sublayer of size $\nu^{3/4}$. Moreover, it is important to note that the instability occurs in the vanishing time $T^\nu$ of order $\sqrt\nu \log\nu^{-1}$.  

In the proof of Theorem \ref{maintheo}, we introduce the following hyperbolic rescaling
$$(T,X,Z) = \frac{1}{\sqrt{\nu}} (t,x,z).$$
Theorem \ref{maintheo} is thus a direct consequence of the scaling and the following theorem, which also addresses the classical stability problem of shear layers in the inviscid limit.

\begin{theo} \label{theoinstable}
Let $U(y)$ be a smooth and analytic function which converges exponentially fast at infinity to a constant, 
with $U(0) = 0$ and assume that it is spectrally unstable for linearized
Euler equations, with a simple eigenvalue. 
Then it is nonlinearly unstable for Navier Stokes equations in $L^\infty$, provided $\nu$ is small enough,
in the following sense.
 For any $s$ arbitrarily large, there exist $\sigma_0 > 0$, $C_0 > 0$ 
and a sequence of solutions $u^\delta$ of Navier Stokes equations \eqref{NS1}-\eqref{NS3} with forcing terms $f^\delta$,
 on some interval $[0,T^\delta]$, such that, as $\delta \to 0$,
$$
\| u^\delta(0) - U^\nu(0) \|_{H^s} \le \delta,
$$
$$
\| f^\delta \|_{L^\infty([0,T^\nu],H^s)} \le \delta,
$$
but
$$
\| u^\delta(T^\delta) - U^\nu(T^\delta) \|_{L^p} \ge \sigma_0, \qquad \forall p\in [1,\infty]
$$
%$$
%\| u^\delta(T^\delta) - U^\nu(T^\delta) \|_{L^1} \ge \sigma_0,
%$$
%$$
%\| u^\delta(T^\delta) - U^\nu(T^\delta) \|_{L^2} \ge \sigma_0
%$$
and
$$
T^\delta = O ( \log \delta^{-1} ) ,
$$
where $U^\nu(t,y)$ is the solution of heat equation with diffusivity $\nu$ and initial data $U(y)$.
\end{theo}

%Note that by interpolation, the instability holds in any $L^p$ with $1 \le p \le \infty$.

Let us now discuss this result. Up to the best of our knowledge it is the first rigorous result of instability of a shear layer
profile near a boundary for Navier Stokes equations. According to Rayleigh's criterium 
the profile $U^\nu$ will have an inflection point.
Physically, this may correspond to a reverse flow and thus rules out the exponential profile $U_\infty (1 - e^{-y / C})$.
The Prandtl equation is well posed for $U^\nu$ and for neighboring analytic profiles. However we do not know whether
the Prandtl equation is well posed for nearby profiles with only Sobolev regularity.

%%%%%

\subsection*{Notations}

%%%%%

For $\alpha \in \rit$ we define
$$
\Delta_\alpha = \partial_y^2 - \alpha^2 
$$
and
$$
\nabla_\alpha = (i \alpha, \partial_y) .
$$
In particular $\nabla_\alpha^2 = (- \alpha^2, \partial_y^2)$.
The three dimensional case is exactly similar to the two dimensional one, therefore we restrict ourselves to the two dimensional case.

%%%%%%%%%%%%%%%%%%%%%%

\section{General strategy}\label{sec-strategy}

%%%%%%%%%%%%%%%%%%%%%%

The proof of Theorem \ref{maintheo} relies on the {\em complete} construction of the instability $u^\nu$.
The first step is to make an isotropic change of variables in $t$, $x$ and $y$; namely, we define
$$
T = {t \over \sqrt{\nu}}, \quad X = {x \over \sqrt{\nu}}, \quad Y = {y \over \sqrt{\nu}} .
$$
Of course, the Navier Stokes equations remain unchanged, except the viscosity which is now $\sqrt\nu$. From now on, we abuse the notation by denoting by $t$, $x$ and $y$ the new variables $T$, $X$ and $Y$.

The starting point is the choice of the shear layer profile (\ref{shear}). We will choose a shear profile $U_0 = (U(y),0)$ 
which is unstable with respect to linearized Euler equations. More precisely, we start from $U_0$, such that
there exists an exponentially growing solution to the following linearized Euler equations
\beq \label{linE1}
\partial_t v + (U_0 \cdot \nabla) v + (v \cdot \nabla) U_0 + \nabla p = 0,
\eeq
\beq \label{linE2}
\nabla \cdot v = 0,
\eeq
\beq \label{linE3} 
v_2 = 0 \qquad \hbox{on} \qquad y = 0 .
\eeq
The study of the linear stability of a shear layer profile is a classical issue in fluid mechanics. The classical strategy to address
this question is to introduce the stream function of $v$ and to take its Fourier transform in the tangential variable $x$ (with dual
Fourier variable $\alpha$) and the Laplace transform in time (with dual variable $\lambda = -i\alpha c$). Precisely, we look for $v$ of the form
\beq \label{stream1}
v = \nabla^\perp \Bigl( e^{i \alpha (x - c t) } \psi(y) \Bigr) + \mbox{complex conjugate.}
\eeq
Putting (\ref{stream1}) in (\ref{linE1}), we get the classical Rayleigh equation for the stream function $\psi$
\beq \label{Ray1}
(U - c) (\partial_y^2 - \alpha^2) \psi = U'' \psi,
\eeq
\beq \label{Ray2}
\alpha  \psi(0) = \lim_{y\to + \infty}  \psi(y) = 0 .
 \eeq
 The study of the linear stability of $U$ reduces to a spectral problem: find $c$ and $\psi$, solutions of Rayleigh equations, with 
 $\Im(\alpha c)  > 0$. Following the classical Rayleigh criterium, if such an instability exists, then $U$ must have an inflection point.
Such smooth unstable profiles do exist (see, for instance, \cite{Grenier00CPAM}).
We choose the most unstable mode, namely the largest $|\alpha \Im c|$.
Starting with such an instability, we can construct an instability for the following linearized Navier Stokes equations
\beq \label{linNS1}
\partial_t v + (U_0 \cdot \nabla) v + (v \cdot \nabla) U_0 - \sqrt \nuÊ\Delta v + \nabla p = 0,
\eeq
\beq \label{linNS2}
\nabla \cdot v = 0,
\eeq
\beq \label{linNS3} 
v = 0 \qquad \hbox{on} \qquad y = 0 .
\eeq
The analogs of the Rayleigh equations \eqref{Ray1}-\eqref{Ray2} are the Orr Sommerfeld equations which read
\beq \label{inOS1}
-\eps (\partial_y^2 - \alpha^2)^2 \psi + (U - c) (\partial_y^2 - \alpha^2) \psi = U'' \psi,
\eeq
\beq \label{inOS2}
\alpha  \psi(0) = \psi'(0) = \lim_{y\to + \infty}  \psi(y) = 0,
 \eeq
 where
$$
\eps = {\sqrt\nu \over i \alpha} .
$$
Such a spectral formulation of the linearized Navier-Stokes equations near a boundary layer shear profile has been 
intensively studied  in the physical literature.  We in particular refer to
\cite{Reid,Sch} for the major works of Heisenberg, Tollmien, C.C. Lin, and Schlichting on the subject. 
We also refer to \cite{GGN1,GGN3,GrN1} for the rigorous spectral analysis of the Orr-Sommerfeld equations.  

Now starting from an unstable mode $(\psi^0,c^0)$ of the Rayleigh equation for some positive $\alpha$, it is possible to construct an unstable mode
$(\psi^\nu,c^\nu)$ for Navier Stokes equations, provided $\nu$ is small enough, such that
\beq \label{mode1}
\psi^\nu - \psi^0 = O(\nu^{1/4}),
\eeq
\beq \label{mode2}
c^\nu - c^0 = O(\nu^{1/4}).
\eeq
Let
$$
\lambda_0 = \alpha c^\nu .
$$
This has been proved rigorously in \cite{GrN1} through a complete analysis of the Green function of Orr Sommerfeld equation.
More precisely, the proof of (\ref{mode1})-(\ref{mode2})
relies on the complete description of all four independent solutions of the fourth order differential
equation (\ref{inOS1}). It can be proven that two of them go to $+ \infty$ as
$y \to + \infty$. These two solutions can be forgotten in the construction of an unstable mode. The other two
converge to $0$ as $y \to + \infty$. One, called $\psi_f$, has a "fast" behavior, namely behaves
like $\exp( - C y / \sqrt{\eps})$ for large $y$. The other one, called $\psi_s$, has a "slow" behavior and behaves
like $\exp(- C | \alpha | y)$. The second one, $\psi_s$, comes from the Rayleigh mode, 
and is a small perturbation of $\psi^0$. Then the unstable mode $\psi^\nu$ is a combination of $\psi_f$ and $\psi_s$
and is of the form
\beq \label{structurepsi}
\psi^\nu = \alpha^\nu \psi_f + \beta^\nu \psi_s.
\eeq
The two relations $\psi^\nu(0) = \partial_y \psi^\nu(0) = 0$ give the dispersion relation. Using the fact
that $\psi_s$ is an approximate eigenmode for Rayleigh equation, (\ref{mode1}) and (\ref{mode2}) can be proved using an implicit function
theorem (see \cite{GrN1} for complete details). We also get, for every positive $k$, that
\beq \label{structurepsi2}
| \partial_y^k \psi^\nu (y) | \le {C_k \over | \eps^{(k - 1) / 2} |Ê}  e^{- C y / | \sqrt{\eps}|Ê} + C_k e^{- \beta y} 
\eeq
for some positive $\beta$.
Note that  $\psi^\nu$ has a boundary layer behavior. This is natural since there is a change of boundary conditions 
between Rayleigh and Orr Sommerfeld equations. The size of the boundary layer is of order $\eps^{1/2} \approx \nu^{1/4}$ (for fixed $\alpha$), which is introduced to
balance $\eps \partial_y^4$ and $\partial_y^2$. This sublayer is known as "viscous sublayer" in the physical literature \cite{Reid}.
Note that $\psi_s$ and $\psi_f$ are analytic on a strip $| \Im y | \le \sigma_0$ for some $\sigma_0 > 0$.

Once the linear instability is constructed, we may construct an approximate solution of the form
\beq \label{uapp}
u^{app}(t,x,y) = \sum_{j=1}^M \nu^{ Nj} u^{j}(t,x,y) 
\eeq
starting from the maximal unstable eigenmode
$$
u^1(t,x,y) = \Re \Bigl( \psi^\nu e^{i \alpha (x - c^\nu t)} \Bigr) .
$$
The construction of such an approximate solution is routine work, and involves successive resolutions of linearized Navier
Stokes equations
\begin{equation}\label{eqs-un}
\begin{aligned}
\partial_t u^n + (U_0 \cdot \nabla) u^n + (u^n \cdot \nabla) U_0 -\sqrt \nu \Delta u^n + \nabla p^n &= F_n,
\\
\nabla \cdot  u^n &= 0,
\end{aligned} \end{equation}
together with the zero initial data and zero Dirichlet boundary conditions, with
$$
F_n = - \sum_{1 \le j \le n-1} (u^j \cdot \nabla ) u^{n-j} .
$$
Note that by construction \cite{Grenier00CPAM}, $u^{app}$ solves Navier Stokes equations, up to a very small term $R^M$, of order
$\nu^{N(M+1)} e^{(M+1) \Re \lambda_\nu t}$, with $\lambda_\nu = -i\alpha c^\nu$, the maximal unstable eigenvalue. 

Let  $u^\nu$ be the solution of Navier Stokes equations
with initial data $u^{app}(0)$. A natural next step is to try to bound the difference $v : = u^\nu - u^{app}$ in $L^2$ norm. However,
we only get
$$
\frac{d}{dt} \| v \|_{L^2}^2 \le C (1+ \| \nabla u^{app} \|_{L^\infty}) \| v \|_{L^2}^2 + \| R^M \|_{L^2}^2 .
$$
As there is a boundary layer in $u^{app}$, $\| \nabla u^{app} \|_{L^\infty}$ is unbounded as $\nu \to 0$, and thus, this energy inequality
is useless when $u^{app} - U$ is of order greater than $\nu^{1/4}$. Using only energy estimates, we cannot obtain $O(1)$ instability
in $L^\infty$, and are limited to $O(\nu^{1/4})$ instability (this is the main limitation of \cite{Grenier00CPAM}).

The reason of this failure is that the viscous sublayer becomes linearly unstable in the inviscid limit \cite{Reid,GGN3}. 
The next natural idea is to work with analytic initial
data and to hope that analyticity will kill sublayer instabilities, exactly as in Caflisch and Sammartino work \cite{SammartinoCaflisch2}, 
where the authors used
analyticity to kill any instability of Prandtl's layers. However in the current setting, we want to get control over time intervals of order
$\log \nu^{-1}$, namely on unbounded time intervals. As the analyticity radius decreases with time, it becomes small, of order
$1 / t$ as $t$ increases, and is therefore too small to control instabilities in large times. This strategy therefore fails.

In this paper, we will directly prove that the series (\ref{uapp}) converges as $M$ goes to $+ \infty$, in analytic spaces.
This leads to a direct construction of a genuine solution of Navier Stokes equations, defined by
\beq \label{utrue}
u^\nu(t,x,y) = U +  \sum_{j=1}^{+ \infty} \nu^{ Nj} u^{j}(t,x,y) 
\eeq
The underlying idea is the following: if we try to control the difference between the true solution and an approximate one, we have
to bound solutions of linearized Navier Stokes equations. However because of the shear, vertical derivatives of such solutions
increase polynomially in time, simply because of the term $\partial_t + U(y) \partial_x$, which generates high normal derivatives.
This polynomial growth can not be avoided, except if we are working with a finite sum of eigenmodes. For eigenmodes, we simply
have an exponential growth, without polynomial disturbances. As a matter of fact, all the terms appearing in
(\ref{utrue}) are driven by eigenmodes through Orr Sommerfeld equations.

The proof of the convergence of (\ref{utrue}) relies on the accurate description of the Green function of Orr Sommerfeld equations,
detailed in \cite{GrN1}, and on the introduction of so called generator functions. Generator functions combine all the norms
of all the $u^j$, and can be seen as a time and space depending norm. We prove that these generator functions satisfy
a Hopf inequality, which allows us to get analytic bounds which are uniform in $M$.

The plan of this paper is the following. We begin with the definition of generator functions. We then study the generator function of
solutions of Laplace equations, and then of Orr Sommerfeld equations. We then detail the construction of $u^j$ and derive uniform
bounds on the generator functions, which ends the proof.

%Note that the profile $U$ will be analytic. It is a trivial solution of Navier Stokes, Prandtl and heat equations, with source term. 
%We do not know whether Prandtl equations are well posed for Sobolev profiles close to $U$ (they are well posed for analytic
%profiles close to $U$). Such a profile $U$ will have inflection points, and possible recirculation domains occur. We do not know how to 
%prove a similar result for a monotonic profile; for instance, for $U(y) = 1 - e^{-y}$.

%%%%%%%%%%%%%%%%%%%%%%%%%

\section{Generator functions}

%%%%%%%%%%%%%%%%%%%%%%%%%

%%%%%

\subsection{Definition}\label{sec-Gen}

%%%%%

Let $f(x,y)$ be a smooth function. 
For $z_1,z_2\ge 0$, we define the following two functions, called in this paper "generator functions"
\beq \label{Genbl}
\begin{aligned}
Gen_0(f)(z_1,z_2) &= \sum_{\alpha \in \ZZ} \sum_{\ell \ge 0}  e^{z_1 |\alpha|}  \|  \partial_y^\ell f_\alpha \|_{\ell,0}{z_2^\ell \over \ell ! } ,
\\Gen_\delta(f)(z_1,z_2) &= \sum_{\alpha \in \ZZ}\sum_{\ell \ge 0}  e^{z_1 |\alpha|} \| \partial_y^\ell f_\alpha \|_{\ell,\delta} {z_2^\ell \over \ell ! } ,
\end{aligned}\eeq
in which $f_\alpha(y)$ denotes the Fourier transform of $f(x,y)$ with respect to the $x$ variable. In these sums, 
$$
\begin{aligned}
\| f_\alpha \|_{\ell,0} &= \sup_{y} \varphi(y)^\ell| f_\alpha(y) | ,
\\\| f_\alpha \|_{\ell,\delta} &= \sup_{y} \varphi(y)^{\ell}| f_\alpha(y) | \Bigl( \delta^{- 1} e^{-y / \delta} + 1 \Bigr)^{-1} ,
\end{aligned}$$
where 
$$
\varphi(y) = \frac{y}{1+y}
$$ 
and where the boundary layer thickness $\delta$ is equal to
$$
\delta = \gamma_0 \nu^{1/4}
$$
for some sufficiently large $\gamma_0>0$. 
More precisely, $\gamma_0 $ will be chosen so that $\gamma_0^{-1} \le \sqrt{\Re \lambda_0/2}$, 
where $\lambda_0$ is the maximal unstable eigenvalue of the linearized Euler equations around $U$.  

Note that $Gen_0$, $Gen_\delta$ and all their derivatives are non negative for positive $z_1$ and $z_2$.
These generator functions $Gen_0(\cdot)$ and $Gen_\delta(\cdot)$ will respectively control the velocity and the vorticity of the solutions
of Navier Stokes equations.

For convenience, we introduce the following generator functions of one-dimensional functions $f = f(y)$:
\beq \label{Genbl-a}
\begin{aligned}
Gen_{0,\alpha}(f)(z_2) &= \sum_{\ell \ge 0}  \|  \partial_y^\ell f \|_{\ell,0}{z_2^\ell \over \ell ! } ,
\\Gen_{\delta,\alpha}(f)(z_2) &= \sum_{\ell \ge 0}   \| \partial_y^\ell f \|_{\ell,\delta} {z_2^\ell \over \ell ! } .
\end{aligned}\eeq
Of course, it follows that 
$$
Gen_0(f) = \sum_{\alpha\in \ZZ} e^{z_1|\alpha|} Gen_{0,\alpha} (f_\alpha)
$$ 
for functions of two variables $f = f(x,y)$, and similarly for $Gen_\delta$.

%%%%

\subsection{Properties}

%%%%

For any $\ell, \ell'\ge 0$, we have
\begin{equation}\label{alg}
\begin{aligned}
\| f\|_{\ell, \delta} \le \|f\|_{\ell, 0},& \qquad \| f\|_{\ell + 1, \delta} \le \|f\|_{\ell, \delta}, \\
\| fg\|_{\ell,\delta} &\le \| f \|_{\ell',0}\|g\|_{\ell-\ell',\delta }.
\end{aligned}\end{equation}
Next, we have the following Proposition 

\begin{prop}\label{prop-Gen}
Let $f$ and $g$ be two functions. For non negative $z_1$ and $z_2$,
there hold
$$
Gen_\delta(f g) \le  Gen_0(f) Gen_\delta(g) ,
$$
$$
Gen_\delta(\partial_x f) = \partial_{z_1} Gen_\delta(f), 
\qquad
Gen_\delta(\partial_x^2 f) = \partial_{z_1}^2 Gen_\delta(f), 
$$
$$Gen_\delta(\varphi \partial_y f) \le C_0 \partial_{z_2} Gen_\delta(f),$$
for some universal constant $C_0$, provided $| z_2 |$ is small enough. 
\end{prop}
\begin{proof} First, note that 
$$ (fg)_\alpha = \sum_{\alpha'\in \ZZ} f_{\alpha'} g_{\alpha - \alpha'} ,$$
and
$$
 \partial_y^\beta (f g)_\alpha =
\sum_{\alpha'\in \ZZ} \sum_{0 \le \beta' \le \beta} 
{\beta ! \over \beta' ! (\beta - \beta') !}  \partial_y^{\beta'} f_{\alpha'} 
\partial_y^{\beta - \beta'} g_{\alpha-\alpha'} .
$$
Thus, 
$$
\begin{aligned}
&Gen_\delta(fg)(z_1,z_2) 
\\&= \sum_{\alpha \in \ZZ}\sum_{\beta \ge 0} e^{z_1|\alpha|}\| \partial_y^\beta (fg)_\alpha \|_{\beta,\delta} {z_2^\beta \over \beta ! } 
\\
&\le \sum_{\alpha \in \ZZ}\sum_{\beta \ge 0} \sum_{\alpha'\in \ZZ}\sum_{0 \le \beta' \le \beta} e^{z_1|\alpha|} \|\partial_y^{\beta'} f_{\alpha'}\|_{\beta',0}
\| \partial_y^{\beta - \beta'} g_{\alpha-\alpha'} \|_{\beta-\beta',\delta}  
{z_2^\beta \over \beta' ! (\beta - \beta') !}  
\\
&\le \sum_{\alpha,\alpha'\in \ZZ}\sum_{ \beta \ge 0} \sum_{\beta \ge \beta' } e^{z_1|\alpha'|} e^{z_1|\alpha-\alpha'|} \|\partial_y^{\beta'} f_{\alpha'}\|_{\beta',0}
\| \partial_y^{\beta - \beta'} g_{\alpha - \alpha'} \|_{\beta-\beta',\delta} 
{z_2^{\beta'}z_2^{\beta - \beta'} \over \beta' ! (\beta - \beta') !}  
\\
&\le Gen_0(f)(z_1,z_2) Gen_\delta (g)(z_1,z_2).
\end{aligned}$$
Next, we write 
$$
\begin{aligned} Gen_\delta(\partial_xf) &= \sum_{\alpha\in \ZZ }\sum_{ \ell \ge 0} e^{z_1|\alpha|}\| \alpha \partial_y^\ell f_\alpha \|_{\ell,\delta} {z_2^\ell \over \ell ! } 
\\&= \partial_{z_1}
\sum_{\alpha\in \ZZ }\sum_{ \ell \ge 0} e^{z_1|\alpha|}\|  \partial_y^\ell f_\alpha \|_{\ell,\delta} {z_2^\ell \over \ell ! } 
= 
\partial_{z_1} Gen_\delta (f),
\end{aligned} $$
and similarly for $Gen_\delta(\partial_x^2 f)$.
Finally, we compute 
$$
\begin{aligned} Gen_\delta(\varphi \partial_yf) 
&= \sum_{\alpha\in \ZZ }\sum_{ \ell \ge 0} e^{z_1|\alpha|}\|  \partial_y^\ell (\varphi \partial_yf_\alpha) \|_{\ell ,\delta} {z_2^\ell \over \ell ! } 
\\
&\le \sum_{\alpha\in \ZZ } \sum_{ \ell \ge 0} 
\sum_{0\le \ell'\le \ell}e^{z_1|\alpha|}\|  \partial_y^{\ell'} \varphi \partial^{\ell - \ell' + 1}_yf_\alpha \|_{\ell ,\delta} {z_2^\ell \over \ell' ! (\ell - \ell')! } 
\\
&\le \Bigl( 1 + \sum_{\ell'\ge 0} \|\partial_y^{\ell'} \varphi\|_{0,0} {z_2^{\ell'} \over \ell' ! }\Bigr)  
 \sum_{\alpha\in \ZZ }\sum_{ \ell -\ell'\ge 0} 
  e^{z_1|\alpha|}\|  \partial^{\ell - \ell' + 1}_yf_\alpha \|_{\ell - \ell'+1,\delta} {z_2^{\ell -\ell'}\over (\ell - \ell')! } 
\\&\le C_0 \partial_{z_2} Gen_\delta (f),
\end{aligned} $$
where we distinguished the cases $\ell' = 0$ and $\ell' > 0$.
As $\varphi$ is analytic,
$\sum_{\ell'\ge 0} \|\partial_y^{\ell'} \varphi\|_{0,0} {z_2^{\ell'} / \ell' ! }$ converges provided $z_2$ is small enough.
The Proposition follows. 
\end{proof}

%%%%%%%%

\subsection{Generator function and divergence free condition}

%%%%%%%%

Note that for any functions $u$ and $g$, Proposition \ref{prop-Gen} yields 
\begin{equation}\label{bd-udxg}
Gen_\delta(u \partial_x g) \le Gen_0(u) \partial_{z_1} Gen_\delta(g) .
\end{equation}
This is not true for $Gen_\delta(v \partial_y g)$, due to the boundary layer weight. We will investigate $Gen_\delta(v \partial_y g)$ 
when $(u,v)$ satisfies the divergence free condition, namely 
$$
\partial_x u + \partial_y v = 0 .
$$
Precisely, we will prove the following Proposition. 
\begin{prop}\label{prop-Gendy} For $| z_2 | \le 1$, there holds
$$
Gen_\delta(v \partial_y g) \le C \Bigl(Gen_0(v) + \partial_{z_1} Gen_0(u) \Bigr) \partial_{z_2} Gen_\delta(g) .
$$
\end{prop}
This Proposition is linked to the deep structure of Navier Stokes equations, namely to the precise link between the transport operator 
and the incompressibility condition. Note that we "loose" one derivative: our bound involves $\partial_x u$. 
\begin{proof}
We compute 
$$
Gen_\delta(v \partial_y g) = \sum_{\alpha\in\ZZ}\sum_{\beta\ge 0} e^{z_1|\alpha|}  
\| \partial_y^\beta ( v \partial_y g)_\alpha \|_{\beta,\delta} 
{z_2^\beta \over \beta ! } ,
$$
in which $$
\partial_y^\beta  ( v \partial_y g)_\alpha
= \sum_{\alpha'\in\ZZ} \sum_{0 \le \beta' \le \beta} 
{\beta ! \over \beta'! (\beta - \beta')!} \partial_y^{\beta'} v_{\alpha'}
 \partial_y^{\beta - \beta' + 1} g_{\alpha-\alpha'}.
$$
For $\beta'>0$, using the divergence-free condition $\partial_yv_\alpha = -i\alpha u_\alpha $, we estimate  
$$\begin{aligned}
\|
\partial_y^{\beta'} v_{\alpha'}\partial_y^{\beta - \beta' + 1} g_{\alpha-\alpha'}\|_{\beta,\delta}
\le \| \alpha' \partial_y^{\beta'-1} u_{\alpha'} \|_{\beta'-1,0}
\| \partial_y^{\beta - \beta' + 1} g_{\alpha-\alpha'}\|_{\beta - \beta' + 1,\delta} .
 \end{aligned}$$
On the other hand, for $\beta' =0$, we estimate 
$$\begin{aligned}
\|
v_{\alpha'}
\partial_y^{\beta  + 1} g_{\alpha-\alpha'}\|_{\beta,\delta}
\le 
\| \varphi^{-1} v_{\alpha'} \|_{0,0}
\|\partial_y^{\beta + 1} g_{\alpha-\alpha'}\|_{\beta  + 1,\delta}
. \end{aligned}$$
We note that for $y\ge 1$, $\varphi(y) \ge 1/2$ and hence 
$$
\| \chi_{\{y\ge 1\}}\varphi^{-1} v_{\alpha'}\|_{0,0} \le 2 \|v_{\alpha'} \|_{0,0}. 
$$ 
When $y\le 1$, using again the divergence-free condition, 
we write $$
v_{\alpha'}(y) = - i\alpha' \int_0^y  u_{\alpha'}(y') dy'  = - i\alpha'  y \int_0^1 u_{\alpha'}(x,\theta y) d\theta.
$$
Therefore, 
$
\varphi(y)^{-1}|v_{\alpha'} (y)| \le \sup_y |\alpha' u_{\alpha'} (y)|$
for $y\le 1$. This proves that 
$$
 \| \varphi^{-1} v_{\alpha'} \|_{0,0} \le 2 \| v_{\alpha'} \|_{0,0} + \| \alpha' u_{\alpha'} \|_{0,0}.
 $$
Combining these inequalities for any $\alpha\in \ZZ$ and $\beta\ge 0$, we obtain 
$$
\begin{aligned}
& \| \partial_y^\beta ( v \partial_y g)_\alpha \|_{\beta,\delta} 
\le \sum_{\alpha'\in \ZZ}
(2 \| v_{\alpha'} \|_{0,0} + \| \alpha' u_{\alpha'} \|_{0,0} )\| \partial_y^{\beta +1} g_{\alpha-\alpha'} \|_{\beta+1,\delta}  
\\&\quad + \sum_{\alpha'\in \ZZ} \sum_{1 \le \beta' \le \beta} \| \alpha'\partial_y^{\beta'-1} u_{\alpha'} \|_{\beta'-1,0}
\| \partial_y^{\beta - \beta' + 1} g_{\alpha-\alpha'}\|_{\beta - \beta' + 1,\delta} {\beta ! \over \beta' ! (\beta - \beta')!}
.\end{aligned}$$
It remains to multiply by $e^{z_1 | \alphaÊ|} z_2^\beta / \beta !$ and to sum all the terms over $\alpha$, $\alpha'$, $\beta$ and
$\beta'$. The second term in the right hand side is bounded by the product of
$$
\sum_\alpha \sum_\beta e^{| \alpha | z_1}  \| \alpha \partial_y^\beta u_{\alpha} \|_{\beta,0} {z_2^{\beta + 1} \over (\beta +1) !},
$$
which is bounded by $Gen_0(\partial_x u)$ provided $| z_2  | \le 1$
and of
$$
\sum_\alpha \sum_\beta e^{| \alpha | z_1} \| \partial_y^{\beta+1} g_\alpha \|_{\beta +1} {z_2^\beta \over \beta !}
,$$
which equals $\partial_{z_2} Gen_\delta(g)$. The first term is similar, which ends the proof.
\end{proof}
Let us now  bound derivatives of the transport term $u \partial_x g + v \partial_y g$.
\begin{prop}\label{prop-Gendy} 
Let
$$
{\cal A} = \Bigl(Id + \partial_{z_1} + \partial_{z_2}\Big) Gen_\delta
$$
and $$
{\cal B} = Gen_0(u) + Gen_0(v) + \partial_{z_1} Gen_0(u) + {\cal A}(g).
$$
Then
$$
{\cal A} (u \partial_x g + v \partial_y g) \le C {\cal B} \partial_{z_1} {\cal B} + C {\cal B} \partial_{z_2} {\cal B} .
$$
\end{prop}
Note that all the terms in ${\cal A}$ are non negative, since all the derivatives of generator functions are non negative.
\begin{proof}
Let us successively bound all the terms appearing in ${\cal A} (u \partial_x g + v \partial_y g)$. First, $Gen_\delta(u \partial_x g + v \partial_y g)$
has been bounded in \eqref{bd-udxg} and in the previous proposition. Next we compute 
\beq \label{Gendx}
\begin{aligned}
\partial_{z_1} Gen_\delta (u \partial_x g) 
&= Gen_\delta(\partial_x (u \partial_x g)) 
= Gen_\delta(\partial_x u \partial_x g + u \partial_x^2 g) 
\\&\le \partial_{z_1} Gen_0(u) \partial_{z_1} Gen_\delta(g) 
+ Gen_0(u) \partial_{z_1}^2 Gen_\delta(g).
\end{aligned}\eeq
Moreover, using Proposition \ref{prop-Gendy}, 
$$
\begin{aligned}
\partial_{z_1} Gen_\delta (v \partial_y g) 
&= Gen_\delta( \partial_x( v \partial_y g))
= Gen_\delta( \partial_x v \partial_y g + v \partial_y \partial_x g) 
\\
&\le C (\partial_{z_1} Gen_0(v) + \partial_{z_1}^2 Gen_0(u)) \partial_{z_2} Gen_\delta(g)
\\&\quad+ C (Gen_0(v) + \partial_{z_1} Gen_0(u)) \partial_{z_2} \partial_{z_1} Gen_\delta(g) .
\end{aligned}$$
Let us now bound the term $\partial_{z_2} Gen_\delta(v \partial_y g)$.
Precisely, we have to bound 
$$
\begin{aligned}
&{z_2^n \over n !} \| \varphi^{n+1} \partial_y^{n+1} (v_{\alpha'} \partial_y g_{\alpha - \alpha'})  \|_{0,\delta}
= {z_2^n \over n !} \| \varphi^{n+1} 
 \partial_y^n (\partial_y v_{\alpha'} \partial_y g_{\alpha - \alpha'}
 + v_{\alpha'} \partial_y^2 g_{\alpha - \alpha'})   \|_{0,\delta} 
\\&\le \sum_{0 \le k \le n} {z_2^n \over k ! (n-k) !}
\| \varphi^{n+1} \partial_y^{k+1} v_{\alpha'} \partial_y^{n+1-k} g_{\alpha - \alpha'} 
+ \varphi^{n+1} \partial_y^k v_{\alpha'} \partial_y^{n+2-k} g_{\alpha - \alpha'} \|_{0,\delta}.
\end{aligned}$$
Let us split this sum in two. The first sum equals, using the divergence free condition,
$$
\begin{aligned}
&\sum_{0 \le k \le n} {z_2^n \over k ! (n-k) !} \| \varphi^k \partial_y^k \partial_x u_{\alpha'} \, \,
\varphi^{n+1-k} \partial_y^{n+1-k} g_{\alpha - \alpha'}  \|_{0,\delta}
\\&\le \sum_{0 \le k \le n} {z_2^k \over k ! } \| \varphi^k \partial_y^k \partial_x u_{\alpha'} \|_{0,0}
{z_2^{n-k} \over (n-k)! } \|\varphi^{n+1-k} \partial_y^{n+1-k} g_{\alpha - \alpha'}  \|_{0,\delta}.
\end{aligned}$$
Multiplying by $e^{| \alpha | z_1}$ and summing over $\alpha$ and $\alpha'$, the sum  is bounded by
$$
Gen_0(\partial_x u) \partial_{z_2} Gen_\delta(g) = \partial_{z_1} Gen_0(u) \partial_{z_2} Gen_\delta(g).
$$
On the other hand, the second sum equals to 
\begin{equation}\label{sum-0}
\sum_{0 \le k \le n} {z_2^n \over k ! (n-k) !} \| \varphi^{n+1} \partial_y^k v_{\alpha'} \, \,
 \partial_y^{n+2-k} g_{\alpha - \alpha'}  \|_{0,\delta}.
\end{equation}
We follow the proof of the previous Proposition. First, for $k > 0$, this sum equals to
$$
\sum_{1 \le k \le n} {z_2^n \over k ! (n-k) !} \| \varphi^{k-1} \partial_y^{k-1} \partial_x u_{\alpha'} \, \,
 \varphi^{n+2 - k}\partial_y^{n+2-k} g_{\alpha - \alpha'}  \|_{0,\delta}.
$$
Multiplying by $e^{| \alpha | z_1}$, the corresponding sum is bounded by
$$
\partial_{z_1} Gen_0(u) \partial_{z_2}^2 Gen_\delta(g),
$$
provided that $| z_2 | \le 1$. It remains to bound the term $k=0$ in \eqref{sum-0}:
$$
{z_2^n \over n !} \| \varphi^{n+1}  v_{\alpha'} \partial_y^{n+2} g_{\alpha - \alpha'}  \|_{0,\delta}
\le \Bigl( 2 \| v_{\alpha'} \|_{0,0} + \| \alpha' u_{\alpha'} \|_{0,0} \Bigr) 
{z_2^n \over n!} \| \varphi^{n+2} \partial_y^{n+2} g_{\alpha - \alpha'} \|_{0,\delta}.
 $$
Multiplying by $e^{| \alpha | z_1}$, the corresponding sum is bounded by
$$
(Gen_0(v) + \partial_{z_1} Gen_0(u) ) \partial_{z_2}^2 Gen_\delta(g) .
$$
This leads to 
$$
\begin{aligned}
\partial_{z_2} Gen_\delta (v \partial_y g) &\le \partial_{z_1} Gen_0(u) \partial_{z_2} Gen_\delta(g) 
\\&\quad + C_0 (Gen_0(v) + \partial_{z_1} Gen_0(u)) \partial_{z_2}^2 Gen_\delta(g). 
\end{aligned}$$
The bound on $\partial_{z_2} Gen_\delta (u \partial_x g)$ is similar
which ends the proof of this Proposition.
\end{proof}

%%%%%%%%%%%%%%

\section{Laplace equations}

%%%%%%%%%%%%%%

In this section we focus on the classical Laplace equation which is much easier than Orr Sommerfeld equations.
We will apply the same arguments on Orr Sommerfeld in the next section.

%%%%%%

\subsection{In one space dimension}

%%%%%%

We consider the classical one-dimensional Laplace equation
\beq \label{Lap1}
\Delta_\alpha \phi = \partial_y^2 \phi - \alpha^2 \phi = f
\eeq
on the half line $y \ge 0$, with  Dirichlet boundary condition $ \phi(0) = 0$ and
$\lim_{y \to + \infty} \phi(y) = 0$. We recall that $\|f \|_{0,0} = \sup_{y \ge 0}|f(y) |$. 
Let us first recall the following classical result:

\begin{prop}\label{proplaplace1} ($L^\infty$ bounds). \\
Let $\phi$ solve the one-dimensional Laplacian problem \eqref{Lap1}, with Dirichlet boundary condition. There holds
\beq \label{Lap3}
\alpha^2 \| \phi \|_{0,0} + | \alpha | \,  \| \partial_y \phi \|_{0,0}
+ \| \partial_y^2 \phi \|_{0,0}  \le C \| f \|_{0,0},
\eeq
where the constant $C$ is independent of the integer $\alpha \ne 0$.
\end{prop}
Note that (\ref{Lap3}) states that we gain two derivatives by inverting the Laplace operator: a control on the maximum of $f$ gives
a control on the maximum of the first two derivatives of $\phi$. 
\begin{proof} 
We will only consider the case $\alpha > 0$, the opposite case being similar.
The Green function of $\partial_y^2 - \alpha^2$ is
$$
G(x,y) =  -{1 \over 2\alpha } \Bigl( e^{- \alpha | x-y |} - e^{- \alpha | x+y |} \Bigr) 
$$
and its absolute value is bounded by $\alpha^{-1} e^{-\alpha|x-y|}$. The solution $\phi$ of (\ref{Lap1}) is explicitly given by
\begin{equation}\label{laplacephi1}
\phi(y) =\int_0^\infty G(x,y) f(x) dx .
\end{equation}
A direct bound leads to
$$
|\phi(y)|\le \alpha^{-1}\|f\|_{0,0}\int_0^\infty e^{-\alpha |x-y|}\; dx \le  C \alpha^{-2} \| f \|_{0,0}
$$
in which the extra $\alpha^{-1}$ factor  is due to the $x$-integration. 
Splitting the integral formula (\ref{laplacephi1}) in $x < y$ and $x > y$ and differentiating it, we get
$$
 \|Ê\partial_y \phi \|_{0,0} \le C \alpha^{-1} \| f \|_{0,0} .
$$
We then use the equation to bound $\partial_y^2 \phi$, which ends the proof of (\ref{Lap3}).
\end{proof}

Next, in the case when $f$ has a boundary layer behavior, we obtain the following result:
 
\begin{prop} \label{proplaplace3} (Boundary layers norms)\\
Let $\phi$ solve the one-dimensional Laplacian problem \eqref{Lap1} with Dirichlet boundary condition. 
Provided $| \delta \alpha^2 | \le 1$, there hold
\beq \label{Lap4}
%| \alpha | \, \| \phi \|_{0,0} 
 \| \nabla_\alpha \phi \|_{0,0}   \le C \| f \|_{0,\delta}
\eeq
and
\beq \label{Lap4bis}
| \alpha |^2 \, \| \phi \|_{0,0} 
+  \| \partial_y^2 \phi \|_{0,\delta}   \le C \| f \|_{0,\delta}
\eeq
where the constant $C$ is independent of the integer $\alpha$.
\end{prop}
Note that in the case of boundary layer norms, we only gain "one" derivative in supremum norm, but the usual two derivatives
in boundary layer norm.
\begin{proof}
Using (\ref{laplacephi1}), we estimate
$$
| \phi (y) | \le \alpha^{-1} \| f \|_{0,\delta} \int_0^\infty
e^{- \alpha |   y -   x |} 
\Bigl( 1 + \delta^{-1} e^{-x/\delta}\Bigr) d  x
$$
$$
\le  \alpha^{-1} \| f \|_{0,\delta} 
\Bigl( \alpha^{-1} + \delta^{-1} \int_0^\infty e^{-x/\delta}d   x \Bigr) 
$$
which yields the claimed bound for $\alpha \phi$. The bound on $\partial_y \phi$ is obtained by differentiating  (\ref{laplacephi1}).

Let us turn to (\ref{Lap4bis}). Note that  $| \partial_x G(x,y) |Ê\le 1$. As
$G(0,y) = 0$ this gives $|G(x,y)| \le | x |$. 
Therefore
$$
| G(x,y) | \le \min(\alpha^{-1}e^{-\alpha | x - y|}, | x|) ,
$$
and hence
$$
| \phi(y) | \le \| f \|_{0,\delta} \int_0^\infty \min(| x | , \alpha^{-1} e^{- \alpha  |x - y|} ) \Bigl( \delta^{-1} e^{- x / \delta} + 1 \Bigr) dx
$$
$$
\le  C \| f \|_{0, \delta} \Bigl(\delta +  \alpha^{-2} \Bigr)
$$
which gives the desired bound when $| \delta \alpha^2 | \le 1$. We then use the equation to get the bound on 
$\| \partial^2_y \phi \|_{0,\delta}$.
\end{proof}

%Alternatively, we can prove this Proposition by splitting the source term $f$ into a boundary layer term $f_b$, bounded
%by $\delta^{-1} e^{- y / \delta}$, and an ``inner'' bounded term, $f_i$.  Indeed, let us first solve
%$$
%\partial_y^2 \phi_a = f_b
%$$
%with $\phi_a(0)= 0$ and with $\phi_a'(0)$ chosen such that $\phi_a$ converges to
%a bounded limit as $y \to + \infty$. We have $| \phi_a | \le C \delta$.
%Let $\tilde \phi = \phi - \phi_a$. Then, $\tilde \phi$ solves the Laplace equation
%$$
%\partial_y^2 \tilde \phi - \alpha^2 \tilde \phi = \alpha^2 \phi_a + f_i.
%$$ 
%with a bounded source, since $\| \alpha^2 \phi_a \|_{0,0} \le C \delta \alpha^2\le C$. 
%Proposition \ref{proplaplace3} follows from Proposition \ref{proplaplace1}.

%%%%%%

\subsection{Laplace equation and generator functions}

%%%%%%

In this section, we will study the generator functions, introduced in Section \ref{sec-Gen}, of solutions to the Laplace equation 
$\Delta\phi = \omega$. In the sequel, it is important to keep in mind that, in the application to Prandtl boundary layer stability,
$\omega$ will  have a boundary layer behavior, namely will behave like
$\delta^{-1} e^{- C y / \delta}$, whereas the stream function $\phi$ will be bounded in the limit. 

\begin{proposition}\label{prop-elliptic}  Let $\phi_\alpha = \Delta_\alpha^{-1}\omega_\alpha$ on $\RR_+$ 
with the Dirichlet boundary condition ${\phi_\alpha}_{\vert_{y=0}} = 0$. For $| \delta \alpha^2 | \le 1$,
 there are positive constants $C_0,\theta_0$ so that 
\beq \label{prop-e}
Gen_{\delta,\alpha}(\nabla_\alpha^2 \phi_\alpha) 
+ Gen_{0,\alpha}(\nabla\phi_\alpha) \le C_0 Gen_{\delta,\alpha}(\omega_\alpha),
\eeq
for all $z_2$ so that $| z_2 | \le \theta_0$.  

Moreover if $\phi = \Delta^{-1} \omega$ and if $\omega_\alpha = 0$ for all $\alpha$ such that $| \delta \alpha^2 | \ge 1$, then
\beq \label{prop-e1}
Gen_0(\nabla \phi) \le C Gen_\delta(\omega),
\eeq
\beq \label{prop-e2}
\partial_{z_1} Gen_0(\nabla \phi) \le C \partial_{z_1} Gen_\delta(\omega),
\eeq
\beq \label{prop-e3}
\partial_{z_2} Gen_0(\nabla \phi) \le C \partial_{z_2} Gen_\delta(\omega) + Gen_\delta(\omega).
\eeq
\end{proposition}

\begin{proof} 
For $n\ge 1$, from the elliptic equation $\Delta_\alpha \phi_\alpha = \omega_\alpha$, we compute  
$$
 \Delta_\alpha ( \varphi^n \partial_y^n \phi_\alpha)  
= \varphi^n \partial_y^n \omega_\alpha + 2 \partial_y (\varphi^n) \partial_y^{n+1} \phi_\alpha
 + \partial_y^2 (\varphi^n) \partial_y^n \phi_\alpha 
.$$
Note that
$$
\partial_y (\varphi^n) \partial_y^{n+1} \phi_\alpha
= n \varphi'  \varphi^{n-1}  \partial_y^{n+1} \phi_\alpha
,$$
and hence the $\| . \|_{0,\delta}$ norm of this term is bounded by $n \| \varphi^{n-1} \partial_y^{n+1} \phi_\alpha \|_{0,\delta}$.
Moreover, 
$$
 \partial_y^2 (\varphi^n) \partial_y^n \phi_\alpha = 
 \Bigl( n (n-1) \varphi'^2 \varphi^{n-2}  + n \varphi'' \varphi^{n-1} \Bigr) \partial_y^n \phi_\alpha
 $$
whose $\| . \|_{0,\delta}$ norm is bounded by $n (n-1) \| \varphi^{n-2} \partial_y^{n} \phi_\alpha \|_{0,\delta}$.
 Using Proposition \ref{proplaplace3}, we get
 $$
\begin{aligned}
&|\alpha|^2 \| \varphi^n \partial_y^n \phi_\alpha\|_{0,0} +  \| \partial_y^2 (\varphi^n \partial_y^n \phi_\alpha) \|_{0,\delta}  
 + \| \nabla_\alpha (\varphi^n \partial_y^n \phi_\alpha)  \|_{0,0}
\\& \le C \| \varphi^n \partial_y^n \omega_\alpha \|_{0,\delta} 
 + C n \| \varphi^{n-1} \partial_y^{n+1} \phi_\alpha \|_{0,\delta}
 + C n (n-1) \| \varphi^{n-2} \partial_y^{n} \phi_\alpha \|_{0,\delta}.
\end{aligned} $$
 Expanding the left hand side, we get
 \beq \label{expand}
\begin{aligned}
&
|\alpha|^2 \|  \varphi^n \partial_y^n \phi_\alpha\|_{0,0} +  \| \varphi^n \partial_y^{n+2} \phi_\alpha \|_{0,\delta}  
 + \| \varphi^n \partial_y^n \nabla_\alpha \phi_\alpha  \|_{0,0}
\\& \le C_0 \| \varphi^n \partial_y^n \omega_\alpha \|_{0,\delta} 
 + C_0n \| \varphi^{n-1} \partial_y^{n+1} \phi_\alpha \|_{0,\delta}
 \\& + C_0n (n-1) \| \varphi^{n-2} \partial_y^{n} \phi_\alpha \|_{0,\delta} 
 + C_0 n \| \varphi^{n-1} \partial_y^n \phi \|_{0,0}.
\end{aligned} 
\eeq
 Let
 $$
 A_n = 
 |\alpha|^2 \|  \varphi^n \partial_y^n \phi_\alpha\|_{0,0} +  \| \varphi^n \partial_y^{n+2} \phi_\alpha \|_{0,\delta}  
 + \| \varphi^n \partial_y^n \nabla_\alpha \phi_\alpha  \|_{0,0}
. $$
Multiplying by $z_2^n / n!$ and summing over $n$, we get
$$
\begin{aligned}\sum_{n \ge 0}ÊA_n {z_2^n \over n ! }
&\le C_0\sum_{n \ge 0}Ê\| \varphi^n \partial_y^{n} \omega_\alpha \|_{0,\delta} {z_2^n \over n ! }
+ C_0 \sum_{n \ge 1}  A_{n-1} {z_2^n \over (n-1) ! }
\\&\quad + C_0 \sum_{n \ge 2}  A_{n-2} {z_2^n \over (n-2) ! }
\end{aligned}$$
which ends the proof of (\ref{prop-e}), provided $| z_2 |$ is small enough.

Next (\ref{prop-e1}) is a direct consequence of (\ref{prop-e}), just summing in $\alpha$. 
If we multiply (\ref{prop-e}) by $| \alpha |$ before summing it, this gives (\ref{prop-e2}).
Now we multiply (\ref{expand}) by $z_2^{n-1} / (n-1)!$ instead of $z_2^n / n!$. This gives
$$
\begin{aligned}\sum_{n \ge 1}ÊA_n {z_2^{n-1} \over (n-1) ! }
&\le C_0\sum_{n \ge 1}Ê\| \varphi^n \partial_y^{n} \omega_\alpha \|_{0,\delta} {z_2^{n-1} \over (n-1) ! }
+ C_0 \sum_{n \ge 1}  A_{n-1} n {z_2^{n-1} \over (n-1) ! }
\\&\quad + C_0 \sum_{n \ge 2}  n (n-1) A_{n-2} {z_2^{n-1} \over (n-1) ! }.
\end{aligned}$$
The terms in the right hand side may be absorbed by the left hand side provided $z_2$ is small enough, %expect
except
$C_0 A_0$, which is bounded by $Gen_{\delta,\alpha} (\omega_\alpha)$. This ends the proof of the Proposition.
\end{proof}

%%%%%%%%%%%

\section{Orr-Sommerfeld equations}

%%%%%%%%%%%

%%%%%%

\subsection{Introduction}

%%%%%%

In this section, we study the Orr-Sommerfeld equations 
\begin{equation}\label{OS1} 
Orr_{\alpha,c}(\phi) :=  - \epsilon \Delta_\alpha^2\phi + (U-c) \Delta_\alpha \phi - U'' \phi = f,
\end{equation}
together with the boundary conditions 
\begin{equation}\label{OS2} \phi_{\vert_{y=0}} = 0, \qquad \partial_y \phi_{\vert_{y=0}} =0,\end{equation}
and $\phi \to 0$ as $y \to + \infty$,
with $\Delta_\alpha = \partial_y^2 - \alpha^2$. We shall focus on the case when $\alpha\not =0$; the $\alpha=0$ case will be treated in Section \ref{sec-azero}. 
 The Orr-Sommerfeld problem is the resolvent problem of the linearized Navier-Stokes equations around
 a shear profile $U$, 
 written in terms of the stream function $\phi$. 
We shall study the generator functions of Orr-Sommerfeld solutions.

Throughout this paper, $| \Im c|$ will always be larger than $3\Re \lambda_0 /2$, 
where $\lambda_0$ is the speed of growth of the linear instability. In particular,
$\Im c$ will be bounded away from $0$. Moreover we will restrict ourselves to $\alpha \ll \epsilon^{-1/3}$, or precisely
%$| \alpha | \le | \eps |^{-1/5}$, namely
\begin{equation}\label{range-a}
| \eps \alpha^3\log \nu | \le 1.
\end{equation}
Let us first describe Orr Sommerfeld equations in an informal way.
For small $\epsilon$, Orr Sommerfeld equations are a viscous perturbation of Rayleigh equations
\begin{equation}\label{def-reRay1} 
Ray_{\alpha,c}(\phi) := (U-c) \Delta_\alpha \phi - U'' \phi = f,
\end{equation}
with boundary conditions $\phi(0) = 0$ and $\lim_{y \to + \infty} \phi(y) = 0$.
Note that the equation $Ray_{\alpha,c}(\phi) = 0$ may be rewritten as
$$
(\partial_y^2 - \alpha^2) \phi - {U'' \over U - c} \phi = 0.
$$
As $y \to + \infty$, it therefore simplifies into $\partial_y^2 \phi - \alpha^2 \phi = 0$. Hence this equation
has two independent solutions $\phi_{s,\pm}$, with respective asymptotic behaviors
$e^{\pm |\alpha| y}$. Note that $c$ is an eigenvalue if and only if $\phi_{s,-}(0) = 0$.
We have to bound these solutions uniformly in $| \alpha | \ge 1$ and $c$, with $| \Im c| \ge 3\Re \lambda_0 / 2$.
As $\alpha$ goes to $\infty$, $\phi_{s,\pm}$ converge to $e^{\pm | \alpha | y}$. Moreover,
$ U'' / (U-c)$ is bounded since $|\Im c| \ge 3 \Re \lambda_0/2$ is bounded away from $0$, therefore this term may be
handled as a regular perturbation. 

When we add the viscous term $\epsilon \Delta_\alpha^2\phi$, these two solutions are slightly perturbed, but give birth to two independent
solutions of Orr Sommerfeld with a "slow" behavior, which behave like $e^{\pm | \alpha | y}$. 
Two additional solutions, called $ \phi_{f,\pm}$, appear, with a fast behavior.
For these solutions the viscous term is no longer negligible and is of the same order as the Rayleigh one. At leading order,
$ \phi_{f,\pm}$ are solutions to
\beq \label{AiryA}
- \eps \partial_y^4 \phi + (U-c + \alpha^2 \epsilon ) \partial_y^2 \phi = 0.
\eeq
Let
$$
\mu_f(y) =  \sqrt{\alpha^2 + \frac{U-c}{\epsilon}},
$$
taking the positive real part.
Then at first order  $\phi_{f,\pm}$ behaves like $e^{\pm \int_0^y \mu_f(z) dz}$.

%
%%%%%%
%
%\subsection{Green function}
%
%%%%%%
%
Let $G_{\alpha,c}(x,y)$ be the Green function of the Orr-Sommerfeld problem. This Green function may be decomposed in
a "slow part" $G_s$ and a "fast part" $G_f$, such that
$$
G_{\alpha,c}(x,y) = G_s(x,y) + G_f(x,y) .
$$
We recall the following theorem, which is the main result of \cite[Theorem 2.1]{GrN1}.

\begin{theorem}\label{theo-GreenOS} 
Let $\alpha, c$ be arbitrary, so that  $| \Im c |$ is bounded away from $0$ and $| \alpha \Im c| > 3 \Re \lambda_0 / 2$.
Then, there are universal positive constants $C_0, \theta_0$ so that 
\begin{equation}\label{est-GrOS1}
 |G_s(x,y)|  \le  \frac{C_0}{\mu_s (1 + |\Im c|)}  \Big( e^{-\theta_0\mu_s |x-y|}  +  e^{-\theta_0 \mu_s |x+y|} \Big) 
 \eeq
 \beq \label{est-GrOS1b}
 | G_f(x,y) | \le  \frac{C_0}{m_f (1 + |\Im c|) } \Big( e^{- \theta_0 m_f|x-y|}  + e^{- \theta_0 m_f|x+y|} \Big) 
\end{equation}
for all $x,y\ge 0$, in which 
 \begin{equation}\label{def-mMf}
\mu_s = |\alpha|, \qquad  m_f = \inf_{y} \Re  \mu_f (y) , % \qquad M_f = \sup_y \Re  \mu_f(y)  
\end{equation}
with 
$$
\mu_f(y) =  \sqrt{\alpha^2 + \frac{U-c}{\epsilon}},
$$ 
taking the positive real part. Similar bounds hold for derivatives, namely for $k \ge 0$ and $l \ge 0$ with $k + l \le 3$,
\begin{equation}\label{est-GrOS2}
 |\partial_x^k \partial_y^l G_s(x,y)|  
 \le  \frac{C_{k,l}}{\mu_s^{1 - k - l}( 1 + |\Im c|)}  \Big( e^{-\theta_0\mu_s |x-y|}  +  e^{-\theta_0 \mu_s (|x+y|)} \Big) 
\eeq
\beq \label{est_GrOS2b}
| \partial_x^k \partial_y^l G_f (x,y) |
\le    \frac{C_{k,l}}{m_f^{1 - k - l} ( 1 + |\Im c| )} \Big( e^{- \theta_0 m_f|x-y|}  + e^{- \theta_0 m_f|x+y|} \Big) 
\end{equation}
%Moreover,
%\begin{equation}\label{est-GrOSLap}
%\begin{aligned}
% |\Delta_\alpha G_{\alpha,c}(x,y)|  &\le  \frac{C_0}{(1+|\Im c|)}  \Big( e^{-\theta_0\mu_s |x-y|}  +  e^{-\theta_0 \mu_s (|x+y|)} \Big) 
% \\&\quad +  \frac{C_0 m_f}{(1+|\Im c|) } \Big( e^{- \theta_0 m_f|x-y|}  + e^{- \theta_0 m_f|x+y|} \Big) .
%\end{aligned}
%\end{equation}
\end{theorem}
%Note that in (\ref{est-GrOSLap}), following (\ref{est-GrOS2}),
%a $\mu_s=|\alpha|$ factor should appear in the denominator of the first right hand side term. However,
%this term vanishes at first order since $\Delta_\alpha e^{\pm |\alpha | y} = 0$. Moreover, the $| \Im c|^{-1}$ factor comes from the fact
%that $\Delta_\alpha \phi$ is of order $f / (U - c)$, namely of order $f / | \Im c|$ for large $|\Im c|$.

%Note that the adjoint equation of Orr Sommerfeld is 
%$$
%- \eps \Delta_\alpha^2 \phi +  \Delta_\alpha \Bigl[ \phi (U - c)  \Bigr] - U'' \phi = f 
%$$
%with the same boundary conditions. The analysis of this adjoint equation is similar to that of Orr Sommerfeld, and the Green function
%of the adjoint operator satisfies similar bounds.

In the sequel, $\Im c$ will always be bounded away from $0$, but can be very large.
This theorem is the main result of \cite[Theorem 2.1]{GrN1}. However, for the sake of completeness
we sketch in the following lines the computation of the Green function, at a formal level.
The Green function $G_{\alpha,c}(x,y)$ is constructed through the representation 
$$
G_{\alpha,c}(x,y) = \left \{ \begin{aligned} 
\sum_{k = s,f} d_{k}(x) \phi_{k,-}(y)  + \sum_{k = s,f} e_k(x) \phi_{k,-}(y) , \quad & y>x>0\\
\sum_{k = s,f} d_{k}(x) \phi_{k,-}(y)  + \sum_{k = s,f} f_k(x) \phi_{k,+}(y), \quad & 0<y<x\\
\end{aligned} \right.  
$$
where the first sum takes care of the boundary condition and the second one of the singularity of the Green function near $x = y$.

%In \cite{GrN1}, these slow and fast modes are constructed as a perturbation of Rayleigh and Airy (\ref{AiryA}) solutions. 
%The slow mode is constructed starting with the solution of the Rayleigh equation.
%For large $\Im c$, the solution of Rayleigh equation can be seen as a perturbation of
%Laplace equation. More precisely, Rayleigh equation may be rewritten
%$$
%- \Delta_\alpha \phi + {U'' \over U - c} \phi = 0,
%$$
%where the second term is a perturbation of the first one. Using the methods of \cite{GrN1} we can show that
%the solution of Rayleigh equation is  at first order $e^{\pm | \alpha | y}$. 
%For fast modes, we showed in \cite{GrN1} that, near $x$, up to a multiplicative constant,
%\begin{equation}\label{soln-bds1}
%\phi_{f,\pm}(y) \approx e^{\pm \mu_f(x) y}.
%\end{equation}
It remains to compute $d_k$, $e_k$ and $f_k$, so that $G_{\alpha,c}$ is continuous, together with its first two derivatives, so that
$\eps \partial^3_y G_{\alpha,c}$ has a unit jump at $y = x$ and so that $G_{\alpha,c}$ satisfies Dirichlet boundary condition together
with its first derivative.
The main contribution in $\eps  \partial^3_y G_{\alpha,c}$ comes from fast modes since $\mu_f \gg \mu_s$. 
In order to get a unit jump at $y = x$
we have to choose, at leading order,
$$
e_f(x)  \sim - {\phi_{f,-}(x)^{-1} \over 2 \eps \mu_f(x)^3}, \qquad
f_f(x) \sim - {\phi_{f,+}(x)^{-1} \over 2 \eps \mu_f(x)^3} 
$$
Note that
$$
{1 \over \eps \mu_f^2} = {1 \over \eps  \alpha^2 + U - c}
$$
and is therefore bounded by $C/ (1+| \Im c|)$ for large $| \Im c |$.
With this choice of $e_f$ and $f_f$, at leading order, $G_{\alpha,c}$ and its second derivatives
are equal at $x^+$ and $x^-$. To match the first derivative we use
$\phi_s(y)$. Note that
$$
\partial_y \Bigl( e_f(x) \phi_{f,-}(y) \Bigr) \sim {1 \over 2 \eps \mu_f^2} = {1 \over 2} {1 \over  \eps \alpha^2 + U - c}
$$
The jump of the first derivative of fast modes is therefore bounded by $C/ (1 + | \Im c|)^{-1}$ for large $| \Im c|$. 
This jump is compensated by the slow modes. 
As a consequence, as the jump in the first derivatives of $e^{\pm | \alpha | y}$ is $|\alpha|$,
 $e_s(x)$ and $f_s(x)$ are of order $1 / |\alpha| (1+| \Im c|)$ for large $| \Im c |$.
The bounds on $d_k$ can be obtained in a similar way.

%%%%%%

\subsection{Pseudoinverse}

%%%%%%

We will also be interested in the case when $Orr_{\alpha,c}$ is not invertible. In this case,
$Im(Orr_{\alpha,c})$ is of codimension $1$, and the equation $Orr_{\alpha,c}(\phi) = f$
may be solved only if $\langle f,\phi_{\alpha,c}\rangle_{L^2} = 0$, where $\phi_{\alpha,c}$ spans the orthogonal
of $Im(Orr_{\alpha,c})$. In \cite{GrN1} we show that we can construct an inverse of $Orr_{\alpha,c}$
on $Im(Orr_{\alpha,c})$ through a kernel $G = G_s + G_f$ which satisfies the same bounds as in Theorem
\ref{theo-GreenOS}. Moreover, the eigenmode of $Orr_{\alpha,c}$ does not lie in
$Im(Orr_{\alpha,c})$. More precisely

\begin{theo} \label{theopseudo} \cite{GrN1}
Let $\alpha$ be fixed.
Let $c_0$ be a simple eigenvalue of $Orr_{\alpha,c}$ with corresponding eigenmode
$\phi_{\alpha,c_0}$. Then there exists a bounded family of linear forms $l^\nu$
and a family of pseudoinverse operators $Orr^{-1}$ such that for any stream function $\phi$,
$$
Orr_{\alpha,c}\Bigl(Orr^{-1}(\phi) \Bigr)  = \phi - l^\nu(\phi) \phi_{\alpha,c_0} .
$$
Moreover, $Orr^{-1}$ may be defined through a Green function $G = G_s + G_f$
which satisfies (\ref{est-GrOS1}) and (\ref{est-GrOS1b}).
\end{theo}

%%%%%%

\subsection{Bounds on solutions of Orr Sommerfeld equations}

%%%%%%

\begin{proposition}\label{prop-OS1} ($L^\infty$ norms) \\ 
Let $\phi$ solve the Orr-Sommerfeld problem \eqref{OS1}-\eqref{OS2}. 
% Let $\lambda_0$ be the maximal unstable eigenvalue of the Rayleigh problem. 
For $| \epsilon \alpha^3 | \le 1$, $|\alpha \Im c| >  3 \Re \lambda_0 / 2$ and $| \Im c|$ bounded away from $0$, there hold
\beq \label{LLa1}
\begin{aligned}
 | \alpha |^2  \|\phi \|_{0,0} + 
 | \alpha |  \| \nabla_\alpha \phi \|_{0,0}  +  \|\nabla_\alpha^2 \phi \|_{0,0}
&\le {C_0 \over 1+ | \Im c |} \|f\|_{0,0}
\end{aligned}
\eeq
and
\beq \label{LLa2}
\begin{aligned}
 \|\sqrt \epsilon \nabla_\alpha^3\phi \|_{0,0}  + \| \epsilon \nabla_\alpha^4\phi \|_{0,0}
&\le C_0 \|f\|_{0,0}.
\end{aligned}
\eeq
\end{proposition}
Equations (\ref{LLa1}) and (\ref{LLa2}) express a classical regularity result: Orr Sommerfeld equation is a small fourth order 
elliptic perturbation
of a second order elliptic equation. Therefore we gain the full control on two derivatives of the solution, and partial controls on
third and fourth derivatives, with prefactors $\sqrt{\eps}$ and $\eps$. 
\begin{proof} By construction, the solution $\phi$ is of the form 
\beq \label{expressint} 
\phi(y) = \int_0^\infty G_{\alpha,c} (x,y) f(x) \; dx .
\eeq
Hence, 
$$
\begin{aligned} 
|\phi(y)| 
&\le  C_0 \int_0^\infty  \Big( \frac{e^{-\theta_0\mu_s |x-y|} }{\mu_s (1+|\Im c|)}
+ \frac{e^{- \theta_0 m_f|x-y|}}{ m_f (1+|\Im c|)}  \Big) f(x)\; dx 
\\
& +  C_0 \int_0^\infty  \Big( \frac{e^{-\theta_0\mu_s |x+y|}}{\mu_s (1+|\Im c|)} 
+ \frac{e^{- \theta_0 m_f|x+y|} }{ m_f (1+|\Im c|)} \Big) f(x)\; dx 
\\
&\le  C_0 \| f \|_{0,0} (\mu_s^{-2} + m_f^{-2}) (1+|\Im c|)^{-1} 
.\end{aligned}$$
We recall that $\mu_s = |\alpha|$ and that  $\epsilon \mu_f^2 = \eps \alpha^2 + (U-c)$. Hence,
$$
{\alpha^2 \over \mu_f^2} = {2 \over 1 + \alpha (U - c)/ (\eps \alpha^3) }.
$$
which is bounded since $| \alpha \Im c | > \Re \lambda_0$ and $| \eps \alpha^3 | \le 1$.
This proves that 
$$
\| \alpha^2 \phi\|_{0,0} \le C_0 (1+|\Im c|)^{-1} \|f\|_{0,0}.
$$ 
To get the bounds on  $\alpha \partial_y\phi$ and $\partial_y^2 \phi$,
 we differentiate (\ref{expressint}) with respect to $y$, splitting the integral in $x < y$ and $x > y$, and fulfill similar computations.

Similarly, we compute 
$$
\begin{aligned} 
|\partial_y^3\phi(y)| 
&\le C_0  (1+|\Im c|)^{-1} \int_0^\infty  \Big( \mu_s^2e^{-\theta_0\mu_s |x-y|} + m_f^2 e^{- \theta_0 m_f|x-y|} \Big) f(x)\; dx 
\\
& + C_0  (1+|\Im c|)^{-1} \int_0^\infty  \Big( \mu_s^2e^{-\theta_0\mu_s |x+y|} + m_f^2 e^{- \theta_0 m_f|x+y|} \Big) f(x)\; dx 
\\
&\le C_0( \mu_s  + m_f) (1+|\Im c|)^{-1} \| f\|_{0,0} 
.\end{aligned}$$
As $\sqrt{| \eps |} |\alpha| \le 1$, $\sqrt{| \eps |} \mu_s \le C$ and
$$
\sqrt{| \eps | } \mu_f = | \sqrt{ \eps \alpha^2 + U - c} | \le C (1 + | \Im c |),
$$ 
which yields the estimate for $\sqrt \epsilon \partial_y^3\phi$.
For $\epsilon \partial_y^4 \phi$, we directly use the Orr-Sommerfeld equation $Orr_{\alpha,c}(\phi) = f$. 
\end{proof}

\begin{proposition}\label{prop-OS2} (Boundary layer norms)\\
Let $\phi$ solve the Orr-Sommerfeld problem \eqref{OS1}-\eqref{OS2}, with source $f$ having a boundary layer behavior.
For $| \epsilon \alpha^3 \log \nu | \le 1$, $| \alpha \Im c | > 3 \Re \lambda_0 / 2$ and $| \Im c|$ bounded away from $0$, there holds
\beq \label{propOS2bis}
 (1 + | \Im c |) \Big( \|  \nabla_\alpha \phi \|_{0,0}  + \| \nabla_\alpha^2 \phi\|_{0,\delta}\Big)
+ | \epsilon | \ \|\nabla_\alpha^4 \phi\|_{0,\delta} \le C_0 \|f\|_{0,\delta}.
\eeq
\end{proposition}
\begin{proof} 
Let $\chi(y)$ be a non negative function which equals $1$ for $0 \le y \le 1$ and $0$ for $y > 1$. Let us split the forcing
term $f$ in its boundary layer term and in its "inner term"
$$
f = f_b + f_i
$$
with 
$$
f_b(y)=   \chi\Bigl( { y \over \delta \log \delta^{-1}} \Bigr) f(y) .
$$
Note that $\| f_i \|_{0,0} \le C \| f \|_{0,\delta}$ and 
$$
| f_b (x) | \le C \| f \|_{0,\delta} \delta^{-1} e^{- y / \delta} .
$$
Let $\phi_b$ and $\phi_i$ be the solutions of $Orr(\phi_b) = f_b$ and $Orr(\phi_i) = f_i$.
Note that $\phi_i$ satisfies (\ref{propOS2bis}), thanks to the previous Proposition. 

It remains to bound $\phi_b$. For this
we split the Green function in its fast part $G_f$ and its slow part $G_s$.
For the fast part we have to bound $G_f \star f_b$, which is a convolution between an exponentially decreasing
kernel and a exponentially decreasing source. It is therefore bounded by
$C \| f \|_{0,\delta} \delta^{-1} e^{- y / \delta}$ provided $m_f > 2 \delta^{-1}$, which is the case provided $\gamma_0$ is large enough.

Let us turn to the slow part $G_s$. Let us first assume that $G_s(0,y) = 0$ for any positive $y$.
Then $\partial_y^2 G_s(0,y) = 0$ for any positive $y$. As
$$
| \partial_y^2 \partial_x G_s(x,y) | \le C {\mu_s^2 \over (1 + | \Im c |)}
$$ 
we have 
$$
| \partial_y^2 G_s(x,y) |Ê\le {C \mu_s^2 x  \over 1 + | \Im c |} ,
$$  
noting the $x$ factor on the right hand side. By convolution between $G_s$ and $f_b$ we have
$$
|\partial_y^2 \phi_b(y) | \le C {| \delta \alpha^2 | \over 1 + | \Im c|}  \| f \|_{0,\delta},
$$
which leads to the desired bound, taking into account that $| \delta \alpha^2 | \le C$, provided that $G_s(0,y) = 0$ for any positive $y$. 

However, it is not the case that $G_s(0,y)=0$ for $y>0$, but we rather have 
$$
G(0,y) = G_f(0,y) + G_s(0,y) =0.
$$ 
Therefore
$\partial_y^2 G_s(0,y) = - \partial_y^2 G_f(0,y)$. For $y \ge \frac{1}{\theta_0 m_f} \log m_f$, we get
$$
|Ê\partial_y^2 G_s(0,y) | \le {C_0  m_f \over  1 + | \Im c | } e^{- \theta_0 m_f y}  \le {C_0  \over  1 + | \Im c | } .
$$
On the other hand, for $y \le  \frac{1}{\theta_0 m_f} \log m_f$, we use
$$
|Ê\partial_y^3 G_s(0,y) | \le {C_0 \mu_s^2 \over 1 + | \Im c | } 
$$
to get
$$
|Ê\partial_y^2 G_s(0,y) | \le {C_0 \over 1 + | \Im c |  } 
+  {C_0 m_f^{-1} \mu_s^2 \log m_f \over 1 + | \Im c | } 
$$
which is bounded by a constant divided by $(1 + | \Im c |)$, upon recalling $m_f > 2 \delta^{-1}$ and using the assumption $\alpha^2 \sqrt \nu \log \frac1\nu \le 1$. This ends the proof of the bound on $\partial_y^2 \phi_b$.
The bounds on $\phi_b$ and $\partial_y \phi_b$ are similar.
\end{proof}

%%%%%%%%%%

\subsection{Generator functions}

%%%%%%%%%%

In this section, we study the generator of solutions to the Orr-Sommerfeld problem.

\begin{proposition}\label{prop-OS3} Let $\phi$ solve the Orr-Sommerfeld problem \eqref{OS1}-\eqref{OS2}, 
with source term $f$.
For $| \epsilon \alpha^3 \log \nu | \le  1$, $| \alpha \Im c | > 3 \Re \lambda_0 / 2$ and for $| \Im c|$ bounded away from $0$, 
there are positive constants $C_0, \theta_0$ (independent
on $\epsilon$ and $\alpha$) so that 
\beq \label{prop-OS3-1}
Gen_{0,\alpha}(\nabla_\alpha \phi) + Gen_{\delta,\alpha}(\nabla^2_\alpha \phi) \le  \frac{C_0}{1+|\Im c|} Gen_{\delta,\alpha} (f),
\eeq
for all $z_2$ so that $0 \le z_2  \le  \theta_0$. Moreover, provided $f_\alpha = 0$ if $| \epsilon \alpha^3  \log \nu | \ge 1$,
\beq \label{prop-OS3-2}
Gen_0(u) + Gen_\delta(\omega)  \le  \frac{C_0}{1+|\Im c|} Gen_\delta (f),
\eeq
\beq \label{prop-OS3-2}
\partial_{z_1} Gen_0(u) + \partial_{z_1} Gen_\delta(\omega) 
\le  \frac{C_0}{1+|\Im c|} \partial_{z_1} Gen_\delta (f),
\eeq
and
\beq \label{prop-OS3-3}
\partial_{z_2} Gen_0(u) + \partial_{z_2} Gen_\delta(\omega) 
 \le  \frac{C_0}{1+|\Im c|} \Bigl[ \partial_{z_2} Gen_\delta(f) + Gen_\delta(f) \Bigr].
\eeq
\end{proposition}
\begin{proof} We estimate each term in the generator functions. 
The term  $n=0$ is already treated in Proposition  \ref{prop-OS2}. 
For $n\ge 1$, we compute 
\begin{equation}\label{OS-yndy}
\begin{aligned}
 Orr_{\alpha,c} (\varphi^n \partial_y^n \phi) 
 &= \varphi^n \partial_y^n f 
- 3 \eps \partial_y \varphi^n \partial_y^{n+3} \phi 
 - 6 \eps \partial_y^2 \varphi^n \partial_y^{n+2} \phi 
\\&\quad 
- 3 \eps \partial_y^3 \varphi^n \partial_y^{n+1}\phi  - \epsilon \partial_y^4\varphi^n \partial_y^n \phi 
\\&\quad + 4 \eps \alpha^2 \partial_y \varphi^n \partial_y^{n+1} \phi
+ 2 \eps \alpha^2 \partial_y^2 \varphi^n \partial_y^n \phi
\\&\quad + 
(U-c)\partial_y^2 \varphi^n \partial_y^n \phi 
+ 2 (U-c)\partial_y \varphi^n \partial_y^{n+1}\phi  
\\&\quad + 
 \sum_{1\le k\le n} \frac{n!}{k! (n-k)!} \varphi^n\Big(\partial_y^k U \partial_y^{n-k}
  \Delta_\alpha \phi - \partial_y^k U'' \partial_y^{n-k} \phi \Big) .
\end{aligned}
\end{equation}
Let us estimate each term on the right. For convenience, we set 
$$ \cA_n: =\| \varphi^n \partial_y^n \nabla_\alpha \phi \|_{0,0}  + \| \varphi^n \partial_y^n \nabla_\alpha^2 \phi\|_{0,\delta} 
+ | \epsilon |  (1 + | \Im c |)^{-1} \| \varphi^n \partial_y^n\nabla_\alpha^4 \phi\|_{0,\delta} ,
$$ 
for $n\ge 0$, and $\cA_n=0$ for negative $n$. As $\varphi = y/(1+y)$, we compute 
$$
\| 3 \eps \partial_y \varphi^n \partial_y^{n+3} \phi 
 + 6 \eps \partial_y^2 \varphi^n \partial_y^{n+2} \phi 
 + 3 \eps \partial_y^3 \varphi^n \partial_y^{n+1}\phi  + \epsilon \partial_y^4\varphi^n \partial_y^n \phi \|_{0,\delta}
 $$
 $$
\le 
C_0(1 + | \Im c |) \sum_{k=1}^4 \frac{n!}{(n-k)!}\cA_{n-k}, 
$$
$$
\| 4 \eps \alpha^2 \partial_y \varphi^n \partial_y^{n+1} \phi
+ 2 \eps \alpha^2 \partial_y^2 \varphi^n \partial_y^n \phi \|_{0,\delta} \le
C \eps \alpha^2 \Bigl[ n \cA_{n-1} + n (n-1) \cA_{n-2}\Big] ,
$$
and 
$$
\begin{aligned}
\|(U-c)\partial_y \varphi^n \partial_y^{n+1}\phi \|_{0,\delta} 
& \le  C_0(1 + | \Im c |) n \cA_{n-1}
\\
\|(U-c)\partial_y^2 \varphi^n \partial_y^n \phi \|_{0,\delta}
& \le  C_0(1 + | \Im c |) \Big[ n \cA_{n-1} + n (n-1) \cA_{n-2}\Big] .
\end{aligned}$$
Finally, we treat the summation in \eqref{OS-yndy}. Set 
$$
\cB_n =  \| \partial_y^nU \|_{0,0} + \|  \partial_y^n U'' \|_{0,0} .
$$
We estimate 
$$
 \sum_{1\le k\le n} \frac{n!}{k! (n-k)!} \|\varphi^n(\partial_y^k U \partial_y^{n-k}
  \Delta_\alpha \phi - \partial_y^k U'' \partial_y^{n-k} \phi ) \|_{0,\delta}
$$
$$
\le C_0 \sum_{1\le k\le n} \frac{n!}{k! (n-k)!} \Big( \|\partial_y^k U\|_{0,0} \|\varphi^{n-k} \partial_y^{n-k}
\Delta_\alpha \phi \|_{0,\delta} 
$$
$$
\qquad  + \| \partial_y^k U'' \|_{0,0}\|\varphi^{n-k}\partial_y^{n-k} \phi  \|_{0,\delta}\Big)
$$
$$
\le C_0 \sum_{1\le k\le n} \frac{n!}{k! (n-k)!} \cB_k \cA_{n-k} .
$$  
Thus, applying Proposition \ref{prop-OS2} to \eqref{OS-yndy}, we obtain, using $| \eps \alpha^2 | \le 1$,  
$$ 
\begin{aligned}
&\|  \nabla_\alpha (\varphi^n \partial_y^n \phi) \|_{0,0}  + \| \nabla_\alpha^2 (\varphi^n \partial_y^n \phi)\|_{0,\delta}
+ | \epsilon | (1 + | \Im c |) ^{-1}\|\nabla_\alpha^4  (\varphi^n \partial_y^n \phi)\|_{0,\delta} 
\\&\le \frac{C_0 }{1 + | \Im c |}\|\varphi^n \partial_y^nf\|_{0,\delta} + C_0 \sum_{k=1}^4 \frac{n!}{(n-k)!}\cA_{n-k} 
+ C_0 \sum_{1\le k\le n} \frac{n!}{k! (n-k)!} \cB_k \cA_{n-k} .
\end{aligned}$$
Expanding the left hand side, we thus have 
$$
\cA_n \le  \frac{C_0 }{1 + | \Im c |}\|\varphi^n \partial_y^nf\|_{0,\delta} + C_0 \sum_{k=1}^4 \frac{n!}{(n-k)!}\cA_{n-k} +
C_0 \sum_{1\le k\le n} \frac{n!}{k! (n-k)!} \cB_k \cA_{n-k} 
$$
for all $n\ge 0$. Multiplying the above equation by $z_2^n / n!$ and summing up the result in $n\ge 0$, we obtain 
\begin{equation}\label{bd-Ann} \begin{aligned}
 \sum_{n\ge 0}\cA_n \frac{z_2^n}{n!}
&\le  \frac{C_0}{1+|\Im c|} Gen_{\delta,\alpha}(f) + C_0  \sum_{n\ge 0}\sum_{k=1}^4 \frac{n!}{(n-k)!}\cA_{n-k}  \frac{z_2^n}{n!}
\\
&\quad + 
C_0  \sum_{n\ge 0} \sum_{k=1}^{n} \frac{n!}{k! (n-k)!} \cB_k \cA_{n-k} \frac{z_2^n}{n!}
.\end{aligned}\end{equation}
Since $|z_2| \le 1$, we compute 
$$ \sum_{n\ge 0}\sum_{k=1}^4 \frac{n!}{(n-k)!}\cA_{n-k}  \frac{z_2^n}{n!} \le 4 z_2  \sum_{n\ge 0}\cA_n \frac{z_2^n}{n!},$$
which can be absorbed into the left hand side of \eqref{bd-Ann}, for sufficiently small $z_2$. Similarly, 
$$
\sum_{n\ge 0} \sum_{k=1}^{n} \frac{n!}{k! (n-k)!} \cB_k \cA_{n-k} \frac{z_2^n}{n!} \le 
C_0   \sum_{n\ge 0}\cA_n \frac{z_2^n}{n!} \sum_{n\ge 1}\cB_{n} \frac{z_2^n}{n!}
$$
which is again absorbed into the left of \eqref{bd-Ann}, upon using the assumption that $U$ is analytic, and
that the sum of $\cB_n$ begins on $nÊ\ge 1$. This ends the proof of (\ref{prop-OS3-1}).

Summing (\ref{prop-OS3-1}) gives (\ref{prop-OS3-2}). Multiplying by $| \alpha |$ before summing
(\ref{prop-OS3-1}) we get (\ref{prop-OS3-3}). Now multiplying by $z_1^{n-1} / (n-1)!$ instead of 
$z_1^n \over n!$ we get
$$
\begin{aligned}
 \sum_{n\ge 1}\cA_n \frac{z_2^{n-1}}{(n-1)!}
&\le  \frac{C_0}{1+|\Im c|} \partial_{z_2} Gen_{\delta,\alpha}(f) 
+ C_0  \sum_{n\ge 1}\sum_{k=1}^4 \frac{n!}{(n-k)!}\cA_{n-k}  \frac{z_2^{n-1}}{(n-1)!}
\\
&\quad + 
C_0  \sum_{n\ge 1} \sum_{k=1}^{n} \frac{n!}{k! (n-k)!} \cB_k \cA_{n-k} \frac{z_2^n}{n!}
.\end{aligned}
$$
If $z_2$ is small enough, the right hand side may be absorbed in the left hand side, excepted 
the terms involving $\cA_0$, which are bounded by $Gen_\delta(f)$. This ends the proof of the Proposition.
\end{proof}

%%%%%%%%%%%%%

\subsection{The $\alpha=0$ case}\label{sec-azero}
In this section, we treat the case when $\alpha=0$. In this case, the resolvent equation of the linearized Navier-Stokes equations simply becomes the resolvent equation for the heat equation  
\begin{equation}
\label{eqs-azero}
(\lambda -\sqrt{\nu }\partial_y^2 ) u_0 = F_0,
\end{equation} 
with $u_0=0$ at $y=0$, where $u_0$ denotes the Fourier mode of the first component of velocity $u$ at $\alpha =0$. We have the following simple proposition. 

\begin{proposition}\label{prop-azero} Let $u_0$ solve \eqref{eqs-azero} and let $\omega_0 = \partial_y u_0$ be the corresponding vorticity. For $\Re \lambda > 3 \Re \lambda_0 / 2$, we have  
	\beq \label{prop-azero1}
	\begin{aligned}
	Gen_{0,0}(u_0) &\le  \frac{C_0}{1+\Re \lambda} Gen_{0,0} (F_0), \\
	Gen_{\delta,0}(\omega_0) &\le  \frac{C_0}{1+\Re \lambda} \Big( Gen_{0,0} (F_0) + Gen_{\delta,0} (\partial_y F_0) \Big) 
	,	\end{aligned}
	\eeq
	for all $z_2\ge 0$.\end{proposition}
\begin{proof} The solution $u_0$ to \eqref{eqs-azero} with the zero boundary condition satisfies 
	$$ u_0(y) = \int_0^\infty G_0(y,z) F_0(z) \; dz$$
	where $G_0(y,z)$ denotes the Green function for $(\lambda - \sqrt \nu \partial_y^2)$ with the Dirichlet boundary condition. In particular, we have 
	$$ |G_0(y,z)| \le \nu^{-1/4}|\lambda|^{-1/2} e^{- \nu^{-1/4} \Re \sqrt \lambda |y-z|} .$$  
In particular, for $\Re \lambda > 3\Re \lambda_0/2$, we estimate 
	$$\begin{aligned}
	 | u_0(y)| 
	 &\le \int_0^\infty \nu^{-1/4}|\lambda|^{-1/2} e^{- \nu^{-1/4} \Re \sqrt \lambda |y-z|} |F_0(z)| \; dz
	 \\
	 &\le C_0 (1+\Re \lambda)^{-1} \sup_y |F_0(y)| .
	\end{aligned}
	$$
The estimates for derivatives are obtained in the same way as done in the previous section for the Orr-Sommerfeld equations. The proposition follows.   
\end{proof}

\section{Construction of the instability}

%%%%%%%%%%%%%

%%%%%%

\subsection{Iterative construction}

%%%%%%

Let us now describe the iterative construction of $u^n$ and $\omega^n$ 
and of the infinite series which defines  the solution \eqref{utrue}. 
We start with the most unstable eigenmode $(\psi_0,\alpha_0,c_0)$ to the Orr-Sommerfeld problem; namely, we start with a solution of 
$$
Orr_{\alpha_0,c_0} (\psi_0) = 0,
$$
with the zero boundary conditions on $y=0$, such that $\alpha_0 \Im c_0$ is maximum. In fact we just need to start
from a mode such that $\alpha_0 \Im c_0$ is strictly larger than half of this maximum. 
Up to a change of sign we may assume
that $\alpha_0 > 0$. Up to a rescaling we may also assume that $\alpha_0 = 1$.

This mode corresponds to a complex solution $\nabla^\perp(e^{i \alpha_0 (x - c_0t)} \psi_0(y))$ of the linearized Navier Stokes equations.
We have to take the real part of this solution in order to deal with real valued solutions. Note that
$(\bar \psi_0,-\alpha_0,\bar c_0)$ is also an eigenmode. We therefore sum up the two unstable eigenmodes
corresponding to $\alpha_0$ and $- \alpha_0$ and define 
$$
\psi^1(t,x,y) =  e^{i \alpha_0 ( x - \Re c_0  t) + \alpha_0 \Im c_0 t} \psi_0(y) 
+  e^{ - i \alpha_0 ( x - \Re c_0  t) + \alpha_0 \Im c_0 t} \bar \psi_0(y)  .
$$
Let $\psi^1_\alpha$ be the Fourier transform of $\psi^1$ in $x$ variable. 
Then all the $\psi^1_\alpha$ vanish, except two of them, namely $\alpha = \pm \alpha_0 = \pm 1$. 
We then iteratively solve the resolvent equation of the linearized Navier-Stokes problem \eqref{eqs-un} for $u^n = \nabla^\perp \psi^n$. In term of vorticity $\omega^n = \Delta \psi^n$, the problem reads
\begin{equation}\label{eqs-wn}
\begin{aligned}
\partial_t \omega^n + U \partial_x \omega^n - U'' \partial_x \psi^n -\sqrt \nu \Delta \omega^n &= - \sum_{1 \le j \le n-1} (u^j \cdot \nabla ) \omega^{n-j},
\end{aligned} \end{equation}
together with the zero boundary condition on $u^n = \nabla^\perp \psi^n$. Precisely, we search for $\psi^n$ 
under the form
\beq \label{firstpsin}
\psi^n = \sum_{|\alpha| \le n} \psi_\alpha^n e^{i \alpha ( x - \Re c_0 t)} e^{n \Im c_0 t},
\eeq
where the sum runs on all the $\alpha$ which are multiples of $\alpha_0$. This yields, for $n\ge 2$,
\begin{equation}\label{OS-un}
Orr_{\alpha,c} (\psi^n_\alpha) = {1 \over i \alpha} \sum_{\alpha'} 
\sum_{1 \le j \le n-1} (u^j_{\alpha - \alpha'} \cdot \nabla_{\alpha'})  \omega_{\alpha'}^{n-j},
\end{equation}
in which $u^j_\alpha = \nabla_\alpha^\perp \psi^j_\alpha$ and $\omega^j_\alpha = \Delta_\alpha \psi^j_\alpha$, 
together with the zero boundary conditions on $\psi^n$ and $\partial_y \psi^n$, and in which
$$
c = \Re c_0  + i n {\alpha_0 \over \alpha} \Im c_0 .
$$
Note that a $\alpha^{-1}$ factor appears in front
of the source term, since Orr Sommerfeld is obtained by taking the vorticity of Navier Stokes equations and dividing by
$\alpha$.
Note also that all but a finite number of $\psi^n_\alpha$ vanish.
Again the sum runs on all the $\alpha$ which are multiple of $\alpha_0$. Note also that
$| \Im c | \ge | \Im c_0|$ and is thus bounded away from $0$.

As Proposition \ref{prop-OS3} only holds for $| \alpha^3 \epsilon |\le 1$ or equivalently
$|\alpha| \le \nu^{-1/4}$, we will only retain the $| \alpha | \le \nu^{-1/4}$ in the construction
of $\psi^n$ and restrict (\ref{firstpsin}) to 
$$
\psi^n = \sum_{| \alpha | \le \nu^{-1/4}} \psi_\alpha^n e^{i \alpha (x - \Re c_0 t)} e^{n \Im c_0 t} .
$$
This leads to the introduction of the force
$$
f^n = \sum_{| \alpha | > \nu^{-1/4}} \sum_{\alpha'} \sum_{1 \le j \le n-1}
(u^j_{\alpha - \alpha'} . \nabla_{\alpha'} ) \omega_{\alpha'}^{n-j} 
$$
that will be estimated below.

%%%%%%%

\subsection{Bounds on $\psi^n$}

%%%%%%%

 We prove the following. 

\begin{proposition}
Introduce the iterative norm 
\begin{equation}\label{def-Gn}
G^n =  Gen_0 (u^n) + Gen_0 (v^n) + \partial_{z_1} Gen_0(u^n) 
\eeq
$$
\qquad \qquad + Gen_\delta(\omega^n) + \partial_{z_1} Gen_\delta(\omega^n) + \partial_{z_2} Gen_\delta(\omega^n),
$$
for $n \ge 1$. Then, $G^n(z_1,z_2)$ are well-defined for sufficiently small $z_1,z_2$, and in addition, 
there exists some universal constant $C_0$ so that 
\beq \label{boundGn}
G^n \le {C \over n} \sum_{1 \le j \le n-1} \Big( G^j \partial_{z_1} G^{n-j} + G^j \partial_{z_2} G^{n-j} \Big) .
\eeq
\end{proposition}
Note that the derivatives appearing in (\ref{def-Gn}) are non negative.

\begin{proof} Applying Proposition \ref{prop-OS3} to the Orr-Sommerfeld equation \eqref{OS-un}, 
using $|\alpha | \le \nu^{-1/4}$,  and summing over $\alpha$, we get 
\begin{equation}\label{iter-un}
{\cal A} (\omega^n) \le {C_0 \over n} {\cal A} \Bigl( \sum_{1 \le j \le n-1}  (u^j \cdot \nabla) \omega^{n-j}  \Bigr) .
\end{equation}
Moreover, using Proposition \ref{prop-elliptic}
$$
Gen_0(u^n) + Gen_0(v^n) \le C Gen_\delta(\omega^n)
$$
and
$$
\partial_{z_1} Gen_0(u^n) \le C \partial_{z_1} Gen_\delta(\omega^n).
$$
Proposition \ref{prop-Gendy} then gives the desired bound.
\end{proof}

%%%%%

\subsection{Bounds on the generator function}

%%%%%

\begin{theo} For $n\ge 1$, let $G^n(z_1,z_2)$ be defined as in \eqref{def-Gn}. Then, the series
$$
G(\tau,z_1,z_2) = \sum_{n=1}^{+\infty} G^n(z_1,z_2) \tau^{n-1}  
$$
converges, for sufficiently small $\tau$, $z_1$, and $z_2$. 
\end{theo}
\begin{proof}
For $N\ge 1$, let us introduce the partial sum 
$$
G_N(\tau,z_1,z_2)  := \sum_{n=1}^N G^n(z_1,z_2) \tau^{n-1} ,
$$
for $\tau,z_1,z_2\ge 0$. 
Note that $G_N$ is a polynomial in $\tau$, and thus well-defined for all times $\tau\ge 0$. 
We also note that all the coefficients $G^n(z_1,z_2)$ are positive. 
In particular, $G_N(\tau, z_1,z_2)$ is positive, and so are all its time derivatives (when $z_1 > 0$ and $z_2 > 0$). 
Moreover, $G_N(\tau,z_1,z_2)$, and all its derivatives, are increasing in $N$. We also observe that, at $\tau =0$,
$$
G_N(0,z_1,z_2) = G^1(z_1,z_2),
$$
for all $N\ge 1$, and hence, 
$$
G(0,z_1,z_2) = \lim_{N\to \infty} G_N(0,z_1,z_2) = G^1(z_1,z_2).
$$ 
Next, multiplying (\ref{boundGn}) by $\tau^{n-2}$ and summing up the result, we obtain 
the following partial differential inequality
$$
\partial_\tau G_N \le C G_N \partial_{z_1} G_N + C G_N \partial_{z_2} G_N ,
$$
for all $N\ge 1$. Therefore the generator function satisfies an Hopf-type equation, or more precisely an Hopf inequality.

As $G_N$ is increasing in $z_1$ and $z_2$, we focus on the diagonal $z_1 = z_2$, and introduce
$$
F_N(\tau,z) = G_N(\tau,\theta(\tau) z, \theta(\tau) z) 
$$
for $\tau, z\ge 0$, where $\theta(\cdot)$ will be chosen later, with $\theta(0)=1$.
It follows that $F_N$ satisfies
\begin{equation}\label{ode-FN}
\partial_\tau F_N \le (2 C F_N  + \theta'(\tau) z) \partial_z F_N .
\end{equation}
Note that $F_N$ is increasing in $N$. At $\tau= 0$, $F_N(0,z) = G_N(0,z,z) = G^1(z,z)$, which is independent on $N$.
Let  $\rho > 0$ be small enough such that
$$
M_0 = \sup_{0 \le z \le \rho} G^1(z,z)
$$
is well defined. 
We now define $\theta(\tau)$ in such a way that 
$$
4 C M_0 + \theta'(\tau) \rho < 0,
$$
with $\theta(0)=1$. For instance, we take $$
\theta(\tau) = 1 - 6 C M_0 \rho^{-1} \tau .
$$
We will work on a time interval where $\theta(\tau) \ge 1/2$, namely on $[0,T_0]$ where $T_0 = \rho / 12 C M_0$.
Let $T_N$ be the largest time $\le T_0$ such that $F_N(\tau ,z)  \le 2 M_0$, for $0 \le \tau \le T_N$ and $0 \le z \le \rho$.
Note that $T_N$ exists and is strictly positive, since $F_N$ is well defined for all the positive times, and continuous in time. It remains to prove that $\inf_{N\ge 1}T_N$ is positive. 

Let us define the characteristics curves $X_N(\tau ,z)$ by solving
$$
\partial_\tau X_N(\tau ,z) = - 2 C F_N(\tau ,X_N(\tau ,z)) - \theta'(\tau) X_N(\tau ,z),
$$
together with $
X_N(0,z) = z
$. Observe that the characteristics are outgoing at $z = 0$ and $z = \rho$. Therefore the characteristics completely fill 
$[0,T_N] \times [0,\rho]$. Let us now introduce 
$$
\widetilde F_N(\tau ,z) = F_N(\tau ,X_N(\tau ,z)) .
$$
It follows from \eqref{ode-FN} that 
$$
\partial_\tau \widetilde F_N(\tau ,z) = \partial_\tau F_N + \partial_\tau X_N \partial_z F_N \le 0.
$$
As a consequence, 
$$
\sup_{0 \le z \le \rho} F_N(\tau ,z) \le \sup_{0 \le z \le \rho} F_N(0,z)  = \sup_{0 \le z \le \rho} G^1(z,z) = M_0.
$$
Therefore $T_N \ge T_0$, for all $N\ge 1$, and $F_N$ is bounded uniformly on $[0,T_0] \times [0,\rho]$. We can therefore take the limit $N \to + \infty$. This leads to the 
convergence of $G_N(\tau,z_1,z_1)$ as $N \to \infty$ for $0 \le \tau \le T_0$ and for $0 \le z_1 \le \rho$. Since $G_N(\tau,z_1,z_2)$ is increasing in $z_1,z_2$, the convergence of $G_N(\tau,z_1,z_2)$ as $N\to \infty$ follows.
\end{proof}

%%%%%%

\subsection{End of proof}

%%%%%%

It remains to bound the force term $f^n$.
For this we note that the cut off occurs for $| \alpha | \ge \nu^{-1/4}$ where the corresponding modes are exponentially
small. The force term is therefore exponentially small itself, and therefore arbitrary small in any Sobolev space.

Now $\sum_n \omega^n \tau^n$ converges for $\tau$ small enough. Let 
$$
\tau = \nu^N e^{\alpha_0 \Im c_0 t} .
$$
Then as long as $\tau$ remains small, namely as long as $t$ remains small than 
$C \log \nu^{-1}$ for some constant $C$, this series converges and defines a solution of the full incompressible
Navier Stokes equations. Note that the series defining $u$ and $v$ also converges in the same way.
This prove Theorem \ref{theoinstable}. Then Theorem \ref{maintheo} follows by a simple rescaling
$(T,X,Z) = \sqrt{\nu}^{-1} (t,x,z)$.

%%%%%%%%%%%

\section{Time dependent shear flow}

%%%%%%%%%%%

We now turn to the case where the shear flow $U_s$ depends on time. Let $\Omega_s$ be the corresponding vorticity.
Note that $U_s$ is a solution of 
$$
\partial_t U_s - \nu \partial_{yy} U_s = 0
$$ 
and hence depends on $t$ through
$\sqrt{\nu} t$. We put this dependency in the notation $U_s(\sqrt{\nu} t,y)$.
The perturbation satisfies
\beq \label{NSnon1}
\begin{aligned}
\partial_t \omega &+ (U_s(0) . \nabla) \omega + (u . \nabla) \Omega_s(0) - \sqrt \nu \Delta \omega 
\\&= - (u . \nabla) \omega
- Q_1(\sqrt{\nu} t) \omega - Q_2(\sqrt{\nu} t) u
\end{aligned}\eeq
where
$$
Q_1(\sqrt{\nu t}) \omega = \Bigl(U_s(\sqrt{\nu t}) - U_s(0) \Bigr) \partial_x \omega
$$
and
$$
Q_2(\sqrt{\nu t}) u = (u . \nabla) \Bigl(\Omega_s(\sqrt{\nu} t) - \Omega_s(0) \Bigr) .
$$
Note that
$$
U_s(\sqrt{\nu t},y) = U_s(0) + \sum_{k = 1}^M (\sqrt{\nu} t)^k U_s^k(y) + O((\sqrt{\nu} t)^{M+1}).
$$
As we are interested in times of order $\log \nu$, and keeping in mind that $\partial_x \omega$ is always bounded
by $\nu^{-1/4}$, we can put the $O()$ in the forcing term.  

We will fulfill a perturbative analysis and look for solutions of (\ref{NSnon1}) of the form
\beq \label{perturb1}
\psi(t,x,y) = \sum_{n \ge 1} \sum_{p \ge 0} \sum_{q \ge 0} \sum_{\alpha \in \zit} e^{n \Im c_0 t} t^p \sqrt{\nu}^q 
e^{i \alpha (x - \Re c_0 t)} \psi_{\alpha}^{n,p,q}(y) .
\eeq
In fact it is sufficient to bound $q$ by some large integer $M$, since we allow a small forcing term.

Putting (\ref{perturb1}) in (\ref{NSnon1}) we get
\beq \label{perturb2}
i\alpha~ Orr_{\alpha,c} (\psi_\alpha^{n,p,q}) + (p+1)  \psi_{\alpha}^{n,p+1,q} = Q^{n,p,q} + L^{n,p,q}
\eeq
where
$$
Q^{n,p,q} = 
\sum_{\alpha'} \sum_{1 \le j \le n-1}Ê\sum_{0 \le k \le p} \sum_{0 \le l \le q}
(u_{\alpha - \alpha'}^{j,k,l} . \nabla_{\alpha'} ) \omega_{\alpha'}^{n-j,p-k,q-l}
$$
and 
$$
L^{n,p,q} = i\alpha
 \sum_{k=1}^M U_s^k \omega_{\alpha}^{n,p-k,q-k}
- i \alpha  \sum_{k=1}^M \psi_\alpha^{n,p-k,q-k} \partial_y \Omega_s^k,
$$
with the convention that a quantity vanishes if one of its indices is negative.
Note that $Q^{n,p,q}$ only involves
$\psi_{\alpha'}^{n',p',q'}$ with $n' < n$. Next, $L^{n,p,q}$ involves $\psi_{\alpha'}^{n',p',q'}$ with $n' \le n$ and $p' < p$ and $q' < q$.

We will solve this equation by recurrence on the power of $\sqrt{\nu}$, namely on $q$.
We begin with the leading order $q = 0$. All the $\psi_\alpha^{n,p,0}$ vanish, except when $p = 0$.
System (\ref{perturb1}) then reduces to (\ref{OS-un}) 
$$
\psi_\alpha^{n,0,0} = \psi_\alpha^n 
$$
which are constructed in the previous section, up to any arbitrarily large $n$. 

We then turn to $q = 1$.
The first terms $\psi_\alpha^{n,0,0}$ create an "error term" $L$ involving $\sqrt{\nu}^k t^k$ for $1 \le k \le M$. 
Let us first focus on the case $k = 1$.
For $k = 1$, the corresponding $L^{n,1,1}$ term is
$$
i\alpha \Bigl( U_s^1 \omega_\alpha^{n,0,0} -\psi_\alpha^{n,0,0} \partial_y \Omega_s^1  \Bigr) . %e^{n \Im c_0 t} .
$$
For $n = 1$ and $\alpha = \pm 1$ we note that $Orr_{\alpha,c}$ is not invertible.
This operator may also be non invertible for other values of $\alpha$ (in finite number).
To simplify the discussion we assume that this does not occur (the general case is similar). 

If $(n,\alpha) \ne (1, \pm 1)$, we take $\psi_\alpha^{n,p,1} = 0$ for $p > 1$. This leads to
\beq \label{perturb10}
i\alpha Orr_{\alpha,c}(\psi_\alpha^{n,1,1}) = 
\sum_{\alpha'} \sum_{1 \le j \le n-1} \sum_{k=0,1}
(u^{j,k,k}_{\alpha - \alpha'} . \nabla_{\alpha'}) \omega_{\alpha'}^{n-j,1-k,1-k} + L^{n,1,1}
\eeq
which is a linearized version of (\ref{OS-un}).

For $(n,\alpha) = (1,\pm 1)$, we note that $Orr_{\alpha,c}$ is not invertible. 
Let $A$ be the right hand side of (\ref{perturb10}). Thanks to Theorem \ref{theopseudo} 
we define 
$$
\psi_{\pm 1}^{1,1,1} = Orr^{-1} (A),
$$ 
and we use $\psi_{\pm 1}^{1,2,1}$ to handle the remainder
$$
2 \omega_{\pm 1}^{1,2,1} =  l^{\nu} (A) \phi_{1,c_0}.
$$
Now to bound $\psi_\alpha^{n,1,1}$ we introduce the corresponding generator function $G^1$ and proceed as
in the previous section. This leads to the following inequality
$$
\partial_t G^1 \le G^0 \partial_{z_1} G^1 + G^0 \partial_{z_2} G^1 
+ G^1 \partial_{z_1} G^0 + G^1 \partial_{z_2} G^0 + C \partial_x G^0 + C G^0 .
$$
Using the same arguments as in the previous section, we obtain bound uniform bounds on $G^1$.
The recurrence can be continued using similar arguments.

%%%%%%%%%%%%%%

\bibliographystyle{abbrv}
%\bibliography{book-ref}

\def\cprime{$'$} \def\cprime{$'$}

\end{document}